\documentclass{article}

\usepackage{amsfonts,xr,graphicx,amsmath,amssymb}

\def\C{\mathbb{C}}
\def\S{\mathbb{S}}
\def\Z{\mathbb{Z}}
\def\R{\mathbb{R}}
\def\N{\mathbb{N}}

\def\tilde{\widetilde}
\def\epsilon{\varepsilon}
\def\phi{\varphi}
\def\n{\hfill\break} \def\al{\alpha} \def\be{\beta} \def\ga{\gamma} \def\Ga{\Gamma} \def\ro{\rho}\def\de{\delta} \def\si{\sigma}\def\ze{\zeta}
\def\om{\omega} \def\Om{\Omega} \def\ka{\kappa} \def\la{\lambda} \def\La{\Lambda}
\def\de{\delta} \def\De{\Delta} \def\vph{\varphi} \def\vep{\varepsilon} \def\th{\theta}
\def\Th{\Theta} \def\vth{\vartheta} \def\sg{\sigma} \def\Sg{\Sigma}\def\ups{\upsilon}\def\ups{\upsilon}
\def\bendproof{$\hfill \blacksquare$} \def\wendproof{$\hfill \square$}
\def\holim{\mathop{\rm holim}} \def\span{{\rm span}} \def\mod{{\rm mod}}
\def\rank{{\rm rank}} \def\bsl{{\backslash}}
\def\il{\int\limits} \def\pt{{\partial}} \def\lra{{\longrightarrow}}
\def\noth{\varnothing}
\def\pa{\partial }
\def\ra{\rightarrow }
\def\sm{\setminus }
\def\ss{\subset }
\def\ee{\epsilon }
\def\beq{\begin{equation}}
\def\eeq{\end{equation}}
\def\ov{\over}
\def\ti{\tilde}
\def\div{\mbox{div}}
\def\supp{\mbox{supp}}
\def\tr{\mbox{tr}}
\def\const{\mbox{const}}

\newtheorem{theorem}{Theorem}[section]
\newtheorem{lemma}[theorem]{Lemma}
\newtheorem{proposition}[theorem]{Proposition}

\numberwithin{equation}{section}

\begin{document}\title{{ Navier-Stokes Equations in Complex Space }}
\author{{Nikolai Nadirashvili\thanks{Aix Marseille Universit\'e, CNRS, Marseille, France, nnicolas@yandex.ru } }}

\bigskip

\date{}
\maketitle
\def\n{\hfill\break} \def\al{\alpha} \def\be{\beta} \def\ga{\gamma} \def\Ga{\Gamma} \def\ro{\rho}\def\de{\delta} \def\si{\sigma} 
\def\om{\omega} \def\Om{\Omega} \def\ka{\kappa} \def\la{\lambda} \def\La{\Lambda}
\def\de{\delta} \def\De{\Delta} \def\vph{\varphi} \def\vep{\varepsilon} \def\th{\theta}
\def\Th{\Theta} \def\vth{\vartheta} \def\sg{\sigma} \def\Sg{\Sigma}\def\ups{\upsilon}\def\ups{\upsilon}
\def\bendproof{$\hfill \blacksquare$} \def\wendproof{$\hfill \square$}
\def\holim{\mathop{\rm holim}} \def\span{{\rm span}} \def\mod{{\rm mod}}
\def\rank{{\rm rank}} \def\bsl{{\backslash}}
\def\il{\int\limits} \def\pt{{\partial}} \def\lra{{\longrightarrow}}
\def\noth{\varnothing}
\def\pa{\partial } 
\def\ra{\rightarrow }
\def\sm{\setminus }
\def\ss{\subset }
\def\ee{\epsilon }
\def\beq{\begin{equation}}
\def\eeq{\end{equation}}
\def\ov{\over}
\def\ti{\tilde}
\def\div{\mbox{div}}
\def\supp{\mbox{supp}}
\def\tr{\mbox{tr}}
\def\const{\mbox{const}}

\bigskip
\bigskip

 {\em Abstract.} We prove global in time regularity of solutions of the Navier-Stokes equations defined in the complex space.
\bigskip


\tableofcontents
  
\section{Introduction}

We consider a motion of incompressible fluid in 3-dimensional space. The equation of motion

 \begin{equation}\label{NS}  {\pa w \over \pa t } - \nu \De w + w \cdot \nabla w -  \nabla p = f       \eeq
with the incompressibility condition
 \begin{equation}\label{div} \div\, w =\sum_{i=1}^3w^i_i=0
\eeq
is known as Navier-Stokes equations.   We assume that solutions of \eqref{NS} \eqref{div} are defined in $\R^3$ at the time interval $t\in (0,T)$. 
External force $f=f(x,t)$ is a smooth divergence free vector field.   
We will assume  that for any $t$ $w(x,t)$ is a periodic vector field defined on $\R^3 $ with the lattice of periods $a \Z^3, \, a>0$, in other words, we assume that
the solutions of \eqref{NS} \eqref{div} are defined  on $M\times (0,T)$, where $M$ is a torus $M=\R^3 / a \Z^3 $. For the system \eqref{NS} \eqref{div} 
we consider  solutions of the Cauchy problem 

 \begin{equation}\label{CP} w(x,0) =w^0(x)  \eeq
where $w^0$ is a divergence free $a \Z^3$ periodic vector field defined on $\R^3$.

For the initial data $w^0$ in $ L^2( \R^3) $  J. Leray in the fundamental work [Le] defined weak solutions to the Cauchy problem \eqref{NS}, \eqref{div},
\eqref{CP} and proved their existence. E. Hopf in [H] gave an alternative proof of the Leray's result, based on the Fadeo-Galerkin approximations to  solutions of the Navier-Stokes equations. Hopf's proof holds for general initial-boundary value problems, including the periodic boundary conditions.

Denote by $D_T$ the space of smooth solenoidal vector fields $\phi \in C^\infty (M\times [ 0,T] ) $ such that $\phi ( \cdot, T) =0$.
If $w$ is a sufficiently regular solution of the Navier-Stokes equations \eqref{NS} \eqref{div} in $ (M\times [ 0,T] ) $  then for every $\phi \in D_T$ 
the following equality holds
 
 \begin{equation}\label{WS}  \int_0^T \left(  (w, { \pa \phi \over \pa t } ) - \nu ( \nabla w, \nabla \phi ) - (w \cdot \nabla w, \phi )   \right)  dt = (w^0, \phi (\cdot ,0 ))
- \int^T_0 (f, \phi) dt \eeq

\medskip {\bf Definition. Leray-Hopf weak solutions.  } Let $w^0 \in L^2(M), \, f\in L^2(\Om \times (0,T) )$. A function $w: M\times [0,T) \to \R^3 $ is said to be a weak solution of \eqref{NS}, \eqref{div} if

\smallskip a) $ w\in L^2(0,T)(H^1(M)) \cap L^\infty (0,T) (L^2(M)) $

\smallskip b) $w$ verifies \eqref{WS}

\smallskip c) The energy inequality: for $t\in [0, T] $

$$ || w\cdot, t) ||^2_{L^2(M) } + 2\nu \int^t_0 || \nabla w ( \tau ) ||^2_{L^2(M) } d\tau \leq || w^0||^2_{L^2(M)} + 2\int^t_0 (w(\tau), f(\tau) )d\tau $$

\smallskip d) $ \lim_{t\to 0} || w( \cdot , t) -w^0 ||_{L^2(M)} =0$

\medskip  J. Leray proved the existence of the global, defined for all $t>0$, weak solutions of the Cauchy problem to Navier-Stokes equations.
The following theorem holds 

\begin{theorem} \label{L}  Let $w^0\in L^2(M), \; \div \, w^0 = 0 $, $ f\in L^2(\Om \times (0,T) )$ and for any $T>0, \,  \div \, f = 0$ in the sense of distributions. Then there exists at least one Leray-Hopf weak solution of the problem
\eqref{NS}, \eqref{div}, \eqref{CP} in $M\times [0, \infty )$.
\end{theorem}

Under certain conditions the weak Leray-Hopf solutions are known to be classical ($C^2$) solutions, consequently they are smooth and real-analytic functions,
see [G], [LR]. Our goal is to prove unconditional analyticity of Leray-Hopf weak solutions \eqref{NS} - \eqref{CP}. The analyticity of the solutions requires as
a necessarily condition the analyticity of the external force $f$. We say that $f$ is a uniformly real-analytic function on $M\times (0, \infty )$ if for any $t, \, f$
is a real-analytic function with a radius of analyticity $r(t)$. For any $T>0, \, r(t)$ is bounded from below on $(0, T), \, r(t) > c(T) >0$ and  $f(z, t)$ has a uniformly bounded holomorphic extension  in a slab $|\Im z |< c(T)$ in  $\C^3$, for all $t\in (0, T)$.

\begin{theorem} \label{NN} Assume that $f(x,t)$ is a solenoidal, uniformly real-analytic vector field defined on $M \times (0, \infty)$. Let $w^0\in L^2(M)$ and $\div \, w^0 = 0$ in the sense of distributions. Let $w(x,t)$ be a Leray-Hopf weak solution of \eqref{NS} - \eqref{CP} on $M\times [0, \infty )$. Then for any $t>0, \, w(x,t)$ is a real-analytic function on $M$. If $w^0 \in L^p(M), \, p>3$, then the weak solution $w$ to
\eqref{NS} - \eqref{CP} is unique and the map $t\to w(\cdot , t) $ is continuous as a map from $[0, \infty ) $ to $L^p(M) $.

\end{theorem}

\smallskip {\bf Remark.} For $w^0\in H^1(M)$ uniqueness and analyticity of the weak solution is known to be true on the interval $(0,T) $, where  $T$
is a sufficiently small constant depending on $\nu$ and $||w^0||_{H^1}$. That follows from Leray's ``Théorème de Structure"
and its generalizations, see [G], [LR].

\smallskip In proofs of  regularity results to the Navier-Stokes equations a standard approach  is to gain the regularity step by step: $L^p$ estimates - smoothness -
analyticity.  The analyticity of smooth solutions to the Navier-Stokes equations is known since the work of Masuda [Ma]. The first complete  results on the real-analyticity 
of  solutions to  problem \eqref{NS} - \eqref{CP} with the initial data in $L^p(M), \, p>3$, 
was given by  Grujić and Kukavica [GK]. For further  results see  [LR].

 In this paper we take an inverse approach to prove the regularity of the solutions. We initially consider the Navier-Stokes equations as a system defined
in the complex space $\C^3$. It becomes there an overdetermined system, its holomorphic solutions satisfy additionally  the Cauchy-Riemann equations,
and that help us to prove the desirable regularity of the solutions.

Possible generalizations and extensions of Theorem \ref{NN} and related topics we will discuss in the last section of the paper.

\smallskip 
 
{\it Acknowledgement. } We would like to thank Z. Grujic  and S. Vladuts  for very important discussions and remarks. 
 
\section{Complex Navier-Stokes equations }

Denote by $X$ and $Y$ the real and imaginary subspaces of $\C^3, \, B_r\ss\R^3=Y, \, r>0$  be a ball $|y| <r, \, M_s$ be a three-dimensional torus
$M = X / a \Z^3$,

$$ \Omega_r =   \{ (x+iy) \in \C^3:  x\in M , y\in B_r\}$$

Assume that a solution $w=( w_1,w_2,w_3)$  of \eqref{NS} - \eqref{CP} is defined on $M\times (0, t_0 )$ and has a holomorphic extension to $\Om_r\times ( 0, t_0 ) , \, r>0$.
Then for any $t\in (0,t_0), \, w(z,t), \, z\in \C^3$ is a solenoidal holomorphic vector field satisfying the complex Navier-Stokes equations, that is, for  any $t\in (0, t_0)$  $w$
satisfies  in  $\Om_r$ Cauchy-Rieman, divergence free and complex Navier-Stokes equations:

\beq \label{CR} { \partial w_i\over  \pa \bar z_j }=0 ,  \eeq
\beq \label{CDiv} \sum_{i=1}^3 {\partial w_i\over  \pa  z_i }=0 ,  \eeq
 \begin{equation}\label{CNS}  {\pa w_i \over \pa t } - \nu \sum^3_{j=1} { \pa^2 w_i \over \pa z^2_j } + \sum^3_{ j=1} w_j { \pa w_i \over \pa z_j } - 
 {\pa  p \over \pa z_i} = F ,       \eeq
 where $p( z, t)$ is a holomorphic function of $z$, $F(z,t)$ is a solenoidal holomorphic vector field in $\Om_r, \, F=f+ig$ and $g=0$ on the real subspace. Notice that if a solution of the real Navier-Stokes equations   are real-analytic functions they have 
 extensions to a complex domain which automatically satisfy equation \eqref{CR}, \eqref{CDiv}, \eqref{CNS}.
 
 Let $U(3)$ be the group of unitary transformations of $\C^3$. The stabilizer in $U(3)$ of real and imaginary subspaces of $\C^3$ is  the orthogonal 
 subgroup $O(3)$ of $U(3)$. Since $O(3)$ acts by orthogonal transformations on $X$ the equations \eqref{CR}, \eqref{CDiv} , \eqref{CNS}  are invariant under the action of $O(3)$.

 Denote $u=\Re w,\, v = \Im w ,\, \tilde w = ( \Im z , w), \, \tilde u = \Re \tilde w,\, \tilde v = \Im \tilde w $. Let $\zeta_j, \, \ze_j= \xi_j +\sqrt{-1} \eta_j,\, j=1,2,3$ be an orthonormal basis in $\C^3, \, \xi_j=x_j, \, \eta_j = y_j$.
  Let $y \in Y,\, y\neq 0$.   Choose $H_y \in O(3) <U(3)$ be such that $H_y \eta_1 = y/ |y| $. Let map $H_y$ transform the basis $\ze_i$ into $\ze^y_i$: 
  $H_y\ze_i = \ze^y_i =\xi^y_i + \sqrt{-1} \eta^y_i $. Denote by $(w^y_1,w^y_2,w^y_3)$ the vector field $w$ in the basis $\ze^y_i$. Denote
  
  $$ u^y_i = \Re w^y_i, \;\; v^y_i = \Im w^y_i $$
  Then for all $y\in B_r$
  
  $$ \div \, u^y =0, \;\; \div v^y=0 \;\;\; on \;\;\; L_y , $$
  where $L_y = \{ X + y \} $
  
  Multiplying the equation \eqref{CNS} on $\bar w $ and integrating it over $\Om_r$, using Cauchy-Riemann equations \eqref{CR} equations \eqref{CDiv},
  \eqref{CNS} and Stokes formula we get
  the following  ``complex flux energy  identity"  for the solutions of the complex Navier-Stokes equations:

$$ - \frac 12 \int_{\Om_r} { \pa | w|^2 \over \pa t} dxdy + \int_{\Om_r} (uf + vg) dxdy $$
$$=    \int_{B_r} 
\int _{M \times y} \Bigl( - \nu \Re \sum^3_{i,j=1}  \bar w^y_j  { \pa^2 w^y_j \over \pa ( \xi^y_i)^2 }
  - \Re \sum^3_{i=1} \bar w^y_i {\pa p \over \pa \xi^y_i }  $$
$$  + \Re \sum^3_{i,j=1} \bar w^y_iw^y_j{\pa w^y_i \over \pa \xi^y_j } \Bigr) dxdy = \nu  \int_{B_r} 
\int _{M \times y} | \nabla_X w |^2 dxdy $$
$$ -  \int_{B_r} \int _{M \times y} \left( \sum^3_{i=1} {\pa u^y_i \over \pa \xi^y_i } \Re p  +  \sum^3_{i=1} {\pa v^y_i \over \pa \xi^y_i } \Im p      \right) dxdy $$
$$ +   \int_{B_r}  \int _{M \times y} \sum^3_{i,j=1}\left(   u^y_iu^y_j {\pa u^y_i \over \pa \xi^y_j } -    u^y_iv^y_j {\pa v^y_i \over \pa \xi^y_j } 
+    v^y_iv^y_j {\pa u^y_i \over \pa \xi^y_j }  +   v^y_iu^y_j {\pa v^y_i \over \pa \xi^y_j }        \right) dxdy$$  
$$=  \nu  \int_{B_r} 
\int _{M \times y} | \nabla_X w |^2 dxdy $$
$$ +   \int_{B_r}  \int _{M \times y} \sum^3_{i,j=1}\left(   \frac 12 u^y_j {\pa (u^y_i)^2 \over \pa \xi^y_j } 
+    v^y_iv^y_j {\pa v^y_i \over \pa \eta^y_j }  +  \frac 12  u^y_j {\pa (v^y_i )^2 \over \pa \xi^y_j } 
+   u^y_iv^y_j {\pa u^y_i \over \pa \eta^y_j } \right) dxdy$$  
$$= \nu  \int_{\Om_r} 
 | \nabla_X w |^2 dxdy $$
 $$ -  \frac 12 \int_{B_r}  \int _{M \times y} \sum^3_{i,j=1}\left(   | u |^2 {\pa u^y_j\over \pa \xi^y_j } + |v|^2  {\pa u^y_j  \over \pa \xi^y_j } 
+  v^y_j{\pa |v^y_i |^2 \over \pa \eta^y_j }        + v^y_j{\pa |u^y_i |^2 \over \pa \eta^y_j }        \right) dxdy$$  
$$= \nu  \int_{\Om_r}   | \nabla_X w |^2 dxdy  -  \frac 12 \int_{B_r}  \int _{M \times y} \sum^3_{i,j=1} v^y_j {\pa (|u^y_i |^2 + |v^y_i |^2 )\over \pa \eta^y_j }      dxdy$$  
$$=  \nu \int_{\Om_r}   | \nabla_X w |^2 dxdy  -  \frac 12 \int_{B_r}  \int _{M \times y} \sum^3_{i,j=1} v^y_j {\pa (|u^y_i |^2 - |v^y_i|^2 + 2|v^y_i |^2 )   \over \pa \eta^y_j }        dxdy$$ 
$$= \nu  \int_{\Om_r}   | \nabla_X w |^2 dxdy  -  \frac 12 \int_{B_r}  \int _{M \times y} \sum^3_{i,j=1} v^y_j {\pa ( \Re ( w_i)^2 + 2|v^y_i |^2 )   \over \pa \eta^y_j }        dxdy$$ 
$$= \nu   \int_{\Om_r}   | \nabla_X w |^2 dxdy  +  \frac 12 \int_{B_r}  \int _{M \times y} \sum^3_{i,j=1} v^y_j \left({\pa \Im ( w_i)^2   \over \pa \xi^y_j }  
-2 {\pa (   |v^y_i |^2 )   \over \pa \eta^y_j }  \right)      dxdy$$   
  $$= \nu  \int_{\Om_r}   | \nabla_X w |^2 dxdy  -  \frac 12 \int_{B_r}  \int _{M \times y} \sum^3_{i,j=1} \left({\pa v^y_j   \over \pa \xi^y_j } \Im (w^2_i)
- 2{\pa (v^y_j  | v^y_i |^2  )  \over \pa \eta^y_j } \right)      dxdy$$   
  $$= \nu  \int_{\Om_r}   | \nabla_X w |^2 dxdy  +   \int_{B_r}  \int _{M \times y} \sum^3_{j=1} {\pa (v^y_j  | v |^2  )  \over \pa \eta^y_j }     dxdy$$   
  $$= \nu  \int_{\Om_r}   | \nabla_X w |^2 dxdy  +   \int_{M \times y}  \int _{B_r}  \div_Y ( v|v|^2 )  dydx    $$
  $$ = \nu  \int_{\Om_r}   | \nabla_X w |^2 dxdy  -   \int_{\pa \Om_r}   (n, v|v|^2 )  d\si  
  = \nu  \int_{\Om_r}   | \nabla_X w |^2 dxdy  + \frac 1r  \int_{\pa \Om_r}   \tilde v|v|^2   d\si$$  
  where $n$ is an outer normal to $\pa \Om_r $ and $d\si$ is the volume element of the surface. 

\medskip Thus we have the following theorem

\begin{theorem} \label{EI} Let $w$ be a solution of the Navier-Stokes equation \eqref{NS}  defined in $M \times (0, T) $ and having a holomorphic extension in
$\Om_r \times (0,T) $. Then for any $t\in (0,1) $ the following equality holds
$$ - \int_{\Om_r} { \pa | w|^2 \over \pa t} dxdy =  2 \nu \int_{\Om_r}   | \nabla_X w |^2 dxdy  + \frac 2r \int_{\pa \Om_r}   \tilde v|v|^2   d\si
- 2 \int_{\Om_r} (uf + vg) dxdy $$    
\end{theorem} 

 \smallskip  Now multiplying the equation \eqref{CNS} on $ w $ and integrating it over $\Om_r$ we get
  
  $$ - \frac 12 \int_{\Om_r} { \pa | u|^2 \over \pa t} dxdy +  \frac 12 \int_{\Om_r} { \pa | v|^2 \over \pa t} dxdy - \int_{\Om_r} (uf - vg) dxdy $$
  $$=    \int_{B_r} 
\int _{M \times y} \Bigl( - \nu \Re \sum^3_{i,j=1}   w^y_j  { \pa^2 w^y_j \over \pa ( \xi^y_i)^2 }
  - \Re \sum^3_{i=1}  w^y_i {\pa p \over \pa \xi^y_i }  $$
$$  + \Re \sum^3_{i,j=1}  w^y_iw^y_j{\pa w^y_i \over \pa \xi^y_j } \Bigr) dxdy = \nu  \int_{B_r} 
\int _{M \times y} ( | \nabla_X u |^2 - | \nabla_X v|^2 ) dxdy $$
$$ -  \int_{B_r} \int _{M \times y} \left( \sum^3_{i=1} {\pa u^y_i \over \pa \xi^y_i } \Re p  -  \sum^3_{i=1} {\pa v^y_i \over \pa \xi^y_i } \Im p      \right) dxdy $$
$$ +   \int_{B_r}  \int _{M \times y} \sum^3_{i,j=1}\left(   u^y_iu^y_j {\pa u^y_i \over \pa \xi^y_j } -    u^y_iv^y_j {\pa v^y_i \over \pa \xi^y_j } 
-   v^y_iv^y_j {\pa u^y_i \over \pa \xi^y_j }  -  v^y_iu^y_j {\pa v^y_i \over \pa \xi^y_j }        \right) dxdy$$  
$$=  \nu  \int_{B_r} 
\int _{M \times y} ( | \nabla_X u |^2 - | \nabla_X v|^2 ) dxdy $$
 $$ -  \frac 12 \int_{B_r}  \int _{M \times y} \sum^3_{i,j=1}\left(   | u |^2 {\pa u^y_j\over \pa \xi^y_j } - |v|^2  {\pa u^y_j  \over \pa \xi^y_j } 
+  v^y_j{\pa |v^y_i |^2 \over \pa \eta^y_j }        - v^y_j{\pa |u^y_i |^2 \over \pa \eta^y_j }        \right) dxdy$$  
$$=  \nu  \int_{B_r} 
\int _{M \times y} ( | \nabla_X u |^2 - | \nabla_X v|^2 ) dxdy -  \frac 12 \int_{B_r}  \int _{M \times y} \sum^3_{i,j=1} v^y_j {\pa (|u^y_i |^2 - |v^y_i |^2 )\over \pa \eta^y_j }      dxdy$$  
$$=  \nu  \int_{B_r} 
\int _{M \times y} ( | \nabla_X u |^2 - | \nabla_X v|^2 ) dxdy -  \frac 12 \int_{B_r}  \int _{M \times y} \sum^3_{i,j=1} v^y_j {\pa ( \Re ( w_i)^2 )   \over \pa \eta^y_j }        dxdy$$ 
$$=  \nu  \int_{B_r} 
\int _{M \times y} ( | \nabla_X u |^2 - | \nabla_X v|^2 ) dxdy +  \frac 12 \int_{B_r}  \int _{M \times y} \sum^3_{i,j=1} v^y_j \left({\pa \Im ( w_i)^2   \over \pa \xi^y_j }  
\right)      dxdy$$   
  $$=  \nu  \int_{B_r} 
\int _{M \times y} ( | \nabla_X u |^2 - | \nabla_X v|^2 ) dxdy -  \frac 12 \int_{B_r}  \int _{M \times y} \sum^3_{i,j=1} {\pa v^y_j   \over \pa \xi^y_j } \Im (w^2_i)$$   
  $$=  \nu  \int_{\Om_r} 
( | \nabla_X u |^2 - | \nabla_X v|^2 ) dxdy $$ 

\medskip Thus we have the following theorem

\begin{theorem} \label{EI2} Let $w$ be a solution of the Navier-Stokes equation \eqref{NS}  defined in $M \times (0, T) $ and having a holomorphic extension in
$\Om_r \times (0,T) $. Then for any $t\in (0,1) $ the following equality holds
  $$ - \int_{\Om_r} { \pa | u|^2 \over \pa t} dxdy +  \int_{\Om_r} { \pa | v|^2 \over \pa t} dxdy $$
  $$= 2 \nu  \int_{\Om_r} 
( | \nabla_X u |^2 - | \nabla_X v|^2 ) dxdy - 2 \int_{\Om_r} (uf - vg) dxdy $$
\end{theorem}

\medskip The last two theorems hold for a suitable Faedo-Galerkin approximations of the Navier-Stokes equations. We  define Faedo-Galerkin 
approximations such way that they a priori will have extension to $\C^3$ as entire functions and that will give us an advantage in application of
Theorem \ref{EI}.

Since the space of smooth solenoidal real vector fields on torus $M$ is invariant under the action of the Laplacian $-\De$ on $M$ it has a full system
of solenoidal eigenfunctions $v_i, \, i=1,\, 2,\, \dots $ with the eigenvalues $\la_1\leq\la_2\leq \dots, \, 0= \la_1=\la_2=\la_3<\la_4$. Since $v_i$ are trigonometric
polynomials  they have holomorphic extensions to $\C^3$. Thus we may assume that functions $v_i$ are defined in the whole space $\C^3$. Notice,
that from the equality
$$ -\De v_i =\la_i v_i \;\;\; on \;\;\; X  $$
follows that  $v_i$ satisfies on $\C^3$ the equation

 \beq \label{eig}   - \sum^3_{i=1} { \pa^2 v_i \over \pa z^2_i } = \la_i v_i \;\;\; on \;\;\; \C^3 \eeq

Denote
$$P_m = \span \{ v_1, \dots , v_m \} $$
In the basis $v_k$  on $M$ we write the equation of Fadeo-Galerkin approximations to the Navier-Stokes equations,

$$ {\pa w^m \over \pa t } - \nu \De w^m  + w^m \cdot \nabla w^m +  \al  = F^m ,      $$

$$F^m\in P_m, \, w^m \in P_m , \;\;\; \al \perp P_m $$

Let 

$$ \pi_m : L^2 (M) \rightarrow P_m $$
be the orthogonal projection of vector fields on $M$ to $P_m$ in $L^2(M)$. We will assume that the the projection $\pi_m$ acts on functions defined on $\C^3$. 
For a function $\phi (z), z\in \C^3$ define $\pi_m \phi$ such that

$$ \pi_m \phi (z) = \pi_m \phi (\cdot , y ) $$
where $z=x+iy$. From \eqref{eig} follows that for functions defined on $M$ and having holomorphic extension to $\C^3$, projection $\pi_m$
commutes with the operator of holomorphic extension  from $X$ to $\C^3$.  

Since $v_i$ is an orthogonal basis in $L^2$ we have for a solenoidal vector field $\phi $ on $M$ the monotonicity of $L^2$ norms

   \begin{equation}\label{mon} ||  \pi_m \phi ||_{L^2(M)}  \leq || \pi_{m+1} \phi  ||_{L^2(M)} \eeq
    and the convergence
    
       \begin{equation}\label{FGcon}   \pi_m \phi   \rightarrow   \phi  \eeq
       in $L^2(M)$.         Since $ \pi_m \phi $ are partial sums of the Fourier series of $\phi $ the last convergence is pointwise  provided the smoothness
       of function $\phi$.

       

The equation of Fadeo-Galerkin approximations to the Navier-Stokes equations we can rewrite as

 \begin{equation}\label{FG} {\pa w^m \over \pa t } - \nu \De w^m  + \pi_m (w^m \cdot \nabla w^m  ) = F^m \eeq
Then from \eqref{eig} follows that for holomorphic extension of $w_m$ to $\C^3$ the last equations read as
 \begin{equation}\label{FGNS}  {\pa w^m \over \pa t } - \nu \sum^3_{j=1} { \pa^2 w^m \over \pa z^2_j } + \pi_m \left(\sum^3_{ j=1} w^m_j { \pa w^m \over \pa z_j }  
 \right)
  = F^m       \eeq
  
  For the approximation of the solution $w$ of initial value problem \eqref{NS}, \eqref{div}, \eqref{CP} we add to the system \eqref{FGNS} the initial data
  \begin{equation}\label{FGin}  w^m( x,0) = w^0_m (x) = \pi_m (w^0(x) ) \eeq 
  where external force $F^m$ is a holomorphic extension of $f^m = \pi_m (f) $,
  $$ F^m = \pi_m(F) = f^m + ig^m $$
  
The following Proposition on the convergence of Faedo-Galerkin approximations essentially is known, but since some details of its proof depend on a particular
choice of the basis $v_i$ it's difficult to give an explicit reference of the result and we give a short proof of it.  
  
   \begin{proposition}\label{pr} Assume $w^0, f$ are smooth and   $w$ is a smooth solution of the Navier-Stokes equations  \eqref{NS} - \eqref{CP}  in $M\times [0,T]$.
 Let $w^m$ be a solution of \eqref{FGNS}, \eqref{FGin}, $w^m(x,0) = w(x,0)$. Then 
$   w^m \to w $ in $L^\infty H^{1}(M)$ as $m\to \infty$. 
 \end{proposition}
 
 {\bf Proof.}  For any choice of Faedo-Galerkin basis, 
 
  \begin{equation}\label{lim} \lim_{m \to \infty} ||  w^m - w||_{L^2(M \times [0,T])} =0  \eeq
 see, [G], [Te].
For any 
 $t_1, t_2 \in [0,T], \, t_2 < t_1$,
 
 $$ 2 \int_M |w^m ( x, t_1) |^2 dx - 2\int_M |w^m ( x , t_2) |^2 dx  =  - \nu \int^{t_1}_{t_2} \int_M | \nabla w^m ( x , t) |^2 dxdt $$
 $$ - \int^{t_1}_{t_2} \int_M  w^m ( x, t) f^m (x, t) dxdt $$
 
 Thus from \eqref{lim}, follows that there is a constant $C>0$ such that for any $\de >0 $ and  
 $t_1, t_2 \in [\de,T], \,  t_1-t_2 > 2\de$, there is $N \in \N$, such that for $m>N$ there are $t'_1\in (t_1 - \de, t_1) , t'_2 \in ( t_2 -\de , t_2) $ such that
 
  $$ \int_M |w^m ( x, t'_1) - w( x, t'_1)|^2 dx + \int_M |w^m ( x , t'_2) - w(x, t'_2) |^2 dx \leq \de $$
and hence 
$$ \int^{t'_1}_{t'_2} \int_M | \nabla w^m ( x , t) |^2 dxdt  \leq C(t'_1-t'_2) $$
therefore  there is $\tau \in ( t'_1, t'_2) $ such that  
   \begin{equation}\label{wlim}  \int_M | \nabla w^m ( x , t) |^2 dxdt  \leq C \eeq

 Multiplying   \eqref{FG} on $\De w^m$, integrating over $M$ using that  $P^m$ is an invariant subspace of the Laplace operator, we get
 
 $$ \frac 12 {\pa \over \pa t} || \nabla w^m (\cdot , t ) ||^2_{L^2(M)} + \nu || \De w^m( \cdot , t) ||^2_{L^2(M)} = (w^m \cdot \nabla w^m, \De w^m)
 + (f^m, \De w^m) $$
 Thus 
  $$ \frac 12 {\pa \over \pa t} || \nabla w^m (\cdot , t ) ||^2_{L^2(M)} + \nu || \De w^m( \cdot , t) ||^2_{L^2(M)} \leq C \int_M |\nabla w^m |^3dx + \frac C\nu ||f^m ||^2_{L^2(M)} $$

 Denote
 
 $$ y(t) = || \nabla w^m ( \cdot , t) ||^2_{L^2(M)} $$
 
 Then from the Hölder and Sobolev inequalities we get
 
    \begin{equation}\label{SN-5} { \pa y \over \pa t } \leq \nu^{-3} C_1 y^3 + C_1  \eeq
 the similar inequality  in [He], or in [G], section 6. 
 
 Let $y(x)$ be a solution of Riccati equation \eqref{SN-5} define on $(0, x_1)$ and $y(0)=C$. There is a constant $\de >0$ such that $y(x) < 2C$ on the interval $(0, \de )$. Thus
 from inequality  \eqref{wlim} follows that 
 
   \begin{equation}\label{wmH1} ||  w^m  ||_{L^\infty (0,T)H^1(M) } <C \eeq
where constant $C$ is independent from $m$. 

 

 Multiplying   \eqref{FG} on $\De^2 w^m$, integrating over $M$ using that  $P^m$ is an invariant subspace of the Laplace operator, we get
 
  $$ \frac 12 {\pa \over \pa t} || \De w^m (\cdot , t ) ||^2_{L^2(M)} + \nu || \De \nabla w^m( \cdot , t) ||^2_{L^2(M)} $$
$$  \leq C \int_M |\nabla w^m | |\De w^m |^2dx  + \frac C\nu || \nabla f^m ||^2_{L^2(M)} $$
$$ \leq \frac C\vep \int_M |\nabla w^m |^2 dx + C \vep\int_M |\De w^m |^2dx  \int_M |\nabla w^m |^2 dx  + \frac C\nu || \nabla f^m ||^2_{L^2(M)} $$
for any $\vep >0$. Thus from inequality \eqref{wmH1} follows

$$ ||  w^m  ||_{L^\infty (0,T)H^2(M) } <C  $$

From the last inequality, the equation \eqref{FGNS} and \eqref{lim} follow the convergence $w^m \to w$ in $L^\infty (0,T) L^2 (M)$. The last convergence implies the convergence  $w^m \to w $ in $L^\infty H^{1}(M)$ as $m\to \infty$.

 
 

\medskip Let $w^m$ be a solution     \eqref{FGNS}.  For the holomorphic extension of $w^m$ to $\C^3$ denote  $v^m = \Im w^m, \, \tilde v^m= ( \Im z, v^m)$. 
       Since $w^m$ satisfies Cauchy-Riemann equations \eqref{CR}, divergence free \eqref{CDiv} equation,    \eqref{FGNS} equations and since $\pi_m$ commutes
       with the operator of holomorphic extension from $X$ to $\C^3$, we can apply to $w^m$ the same arguments as in the proofs of Theorems \ref{EI}, \ref{EI2}
 and we get ``complex flux energy  identity" and Theorem \ref{EI2}  for the field $w^m$

\medskip 

\begin{theorem} \label{FGE} Let $w^m$ be a solution of equation \eqref{FG}  defined in $M \times (0, T) $. Then for holomorphic extension of $w^m$  in
$\Om_r \times (0,T) $  the following equality holds
$$ - \int_{\Om_r} { \pa | w^m|^2 \over \pa t} dxdy =  2 \nu  \int_{\Om_r}   | \nabla_X w^m |^2 dxdy  +  \frac 2r \int_{\pa \Om_r}   \tilde v^m|v^m|^2   d\si $$
$$- 2 \int_{\Om_r} (u^mf^m + v^mg^m) dxdy $$    
\end{theorem} 

\medskip 

\begin{theorem} \label{FGE2} Let $w^m$ be a solution of equation \eqref{FG}  defined in $M \times (0, T) $. Then for holomorphic extension of $w^m$  in
$\Om_r \times (0,T) $  the following equality holds
  $$ - \int_{\Om_r} { \pa | u^m|^2 \over \pa t} dxdy +   \int_{\Om_r} { \pa | v^m|^2 \over \pa t} dxdy = 2 \nu  \int_{\Om_r} 
( | \nabla_X u^m |^2 - | \nabla_X v^m|^2 ) dxdy $$ 
$$ - 2 \int_{\Om_r} (u^mf^m - v^mg^m) dxdy $$ 
\end{theorem}

\section{Functional spaces. Hypoelliptic operators}

In this section we collect some auxiliary results.  First we introduce functional spaces
which we will use and then discuss some properties of degenerate elliptic and parabolic equations which we require. 

Let $L^p,\, 1\leq p \leq \infty $, be the standard Lebesgue norm  for functions defined on  spaces with measure. For a function $f$  on $\R^n$ we
define the mixed Lebesgue norm $L^{\bf p}$, the norm when ${\bf p} $ depends on the variable $x_1, \dots ,x_n$: for ${\bf p} =( p_1,\dots , p_n) $ set

$$ || f||_{L^{\bf p} } = || f ||_{L^{p_n}( x_1) \dots L^{p_1} ( x_n)} =  \left( \int_\R \dots \left( \int_\R | f( x_1, \dots , x_n) |^{p_1} dx_1 \right)^{p_2\over p_1} 
\dots dx_n \right)^{1 \over p_n}   $$

\smallskip Let $\si$ be a permutation of the set $\{ 1,2,\dots , n\} $. Denote $$\si (x) = ( x_{\si( 1)} , x_{\si (2)} , \dots , x_{\si ( n)} \} $$ For  a given
vector ${\bf p} $, let $\si_* $ and $\si^*$ be permutations of $\{ 1,2, \dots , n\} $ having components 
of $p_{\si_*} $ and  $p_{\si^*} $ in non-decreasing and non-increasing oder respectively,

$$   p_{\si_*( 1)} \leq p_{\si_* (2)} \leq \dots \leq p_{\si_* ( n)} , $$ 
$$   p_{\si^*( 1)} \geq p_{\si^* (2)} \geq \dots \geq p_{\si^* ( n)}  $$ 

\begin{theorem} \label{MN1} For any permutation $\si$ of  $\{ 1,2, \dots , n\} $ and for any function $f$ of $n$ variables we have
$$ || f( \si_*^{-1} (x) )||_{L^{\si_* {\bf p} } } \leq || f( \si^{-1} (x)) ||_{L^{\si {\bf p} }} \leq || f( \si^{*-1} (x)) ||_{L^{\si^* {\bf p} } } $$ 
\end{theorem} 

\smallskip Theorem \ref{MN1} and the next theorem ``Mixed Hölder inequality" one can find in [AF].

\begin{theorem} \label{MN2} Let ${\bf r } \in [1, \infty )^m,  \, {\bf p }(1) , \dots , {\bf p }(N) \in [1, \infty ]^m$. 

1. If $f_k \in L^{ {\bf p} (k) }, \, k= 1, \dots , N $ and

$$ {1\over {\bf r }_j } = {1\over {\bf p }_j (1)} + \dots + {1\over {\bf p }_j (N)} ,$$ 
$j= 1, \dots ,m$ then
$$ f_1\dots f_N \in L^{\bf r} $$
and
$$|| f_1\dots f_N ||_{L^{\bf r}} \leq || f_1||_{L^{ {\bf p} (1) } }\dots ||f_N||_{L^{{\bf p}(N) }} $$

2. If $\th_1+   \dots + \th_N =1      $

$$ { 1 \over {\bf r}_j } = \sum^N_{k=1}  { \th_k \over {\bf p}_j(k) } ,$$
$j= 1, \dots ,m$ and if $f \in L^{ {\bf p}(k)} $, $k=1, \dots , N$, then $f\in L^{\bf r}$ and

$$ || f||_{L^{\bf r}} \leq || f ||^{\th_1}_{L^{{\bf p}(1)}} \dots || f ||^{\th_N}_{L^{{\bf p}(N)}} $$
\end{theorem} 

\medskip Let $m\in \N,\, 1\leq p \leq \infty$. By $H^{m,p}$ we denote Sobolev space with the norm

$$ || f ||_{H^{m,p} ( \Om)} = \left( \sum_{ |\al | \leq m } || D^\al ||^p_{L^p} \right)^{1/p} $$
for a smooth function $f$ in a domain $\Om \ss \R^n$.

 For Sobolev semi-norm we take the leading derivative,

$$ | f |_{\dot H^{m,p} ( \Om)} = \left( \sum_{ |\al | = m } || D^\al ||^p_{L^p} \right)^{1/p} $$

The definition of $H^{m,p}$ can be extended for all $m\in \R$, see [LM1],  We denote $H^s =H^{s,2}$.












\smallskip  In the sequel we will use the following special partition of unity.

\smallskip {\bf Definition.} Partition of unity of rank $N$.  We say that 
$ \Phi = \{ \phi_1, \phi_2, \dots \} $
is a partition of unity on $\R^n $ (or more general, on a smooth manifold)  of rank $N$,  if $\phi_i ,   \, i=1,2, \dots$ are nonnegative smooth
on $\R^n$ functions, with compact supports, such that

$$ \sum \phi_i \equiv 1 \;\;\; on \;\;\; \R^n $$
and for any $i\in 1,2, \dots $

$$ \sharp \{ j: \phi_i\phi_j \neq  0\} \leq N $$

\smallskip  \begin{proposition} \label{PU} Let $ \Phi = \{ \phi_1, \phi_2, \dots \} $
is a partition of unity on $\R^n $ of rank $N$.  There is a  constant $C = C( N) >0 $ such that for any
smooth function $f$ with a compact support on $\R^n$ the following inequalities hold
$$ \sum^\infty_{i=1} || \phi_i f ||^2_{L^2}     \leq  || f ||^2_{L^2}  \leq C \sum^\infty_{i=1} || \phi_i f ||^2_{L^2}     $$
\end{proposition} 

\smallskip {\it Proof.}  We say that the pair $\{ i,j\} \in S$ if $\phi_i\phi_j \neq 0$. Then

$$  || f ||^2_{L^2}  = \int_{\R^n} | \sum \phi_i f |^2 dx = \int_{\R^n} \sum_{ \{i_1,  i_2\} \in S} \phi_{i_1} \phi_{i_2} f^2 dx  $$
$$ \leq 2 \int_{\R^n} \sum_{\{i_1,  i_2\} \in S} (\phi^2_{i_1} + \phi^2_{i_2}) f^2 dx    \leq 2N\sum^\infty_{i=1} || \phi_i f ||^2_{L^2} $$
Since $\phi_i$ are nonnegative functions we have the left-hand side inequality.

\medskip  Let $\Ga $ be a compact smooth Riemannian manifold. There are several equivalent ways to define Sobolev spaces $H^s$ on $\Ga $.
Following [St], we define spaces $H^s$ using fractional Laplacian. Let $\De $ be the  Laplace-Beltrami operator on $\Ga $. Let
$ \psi_0, \psi_1, \psi_2, \dots ,$ be an orthonormal basis in $L^2(\Ga) $ of  eigenfunctions of $( -\De )$ with the eigenvalues

$$ 0 <\la_1 \leq \la_2 \leq \dots  $$
Then for a smooth function $u$ on $\Ga $, $s\geq 0$,

$$ (-\De )^s u = \sum_{i=1}^\infty \la_k^s (u, \psi_k) \psi_k $$
Correspondingly, for $s\geq 0$ we define the semi-norm

$$ | u |^2_{\dot H^{s} (\Ga )} = ( (- \De )^s u ,u) $$
and the norm

$$ || u||_{H^s(\Ga)} =  | u |^2_{\dot H^{s} (\Ga )}  + || u ||^2_{L^2(\Ga)} $$

 Let $G\ss \R^n$  be a domain with a smooth boundary.
Let $f$ be a function  defined on the boundary $\pa G$ Denote by $V$ the harmonic extension of $f$ to $G$, i.e., $V$ is the solution of the
Dirichlet problem in $G$ with the boundary data $f$. 

The following  theorems  evaluate the Sobolev norms of $f$, by norms of its harmonic  extensions to domain $G$. The theorems
hold for general domains in $\R^n$, however, in the sequel we will use these results for the domains 

$$ G = \{ (x,y) : x \in \R^n, \, y \in \R^m , \, |y|<1 \} ,$$
$$ G' = M_a \times B_1$$ 
Denote $\Ga = \pa G' $,

$$ \ga_t =\{ x\in G, \; dist (x,  \pa G) =t \} ,$$
$$ \Ga_t =\{ x\in G', \; dist (x, \Ga) =t \} ,$$
$t>0$.

\begin{theorem} \label{3.1}  Let $f$ be a smooth function on $\Ga$. Then
$$ || V ||_{\dot H^1(G')} \leq  C|| f ||_{ \dot H^{1/2} ( \Ga)} ,$$
$$ || V ||_{ H^1(G')} \leq C || f ||_{  H^{1/2} (\Ga)} ,$$
and if $f$ and consequently $V$ are odd functions of $y$, then
$$ C|| V ||_{\dot H^1(G')} \geq  || f ||_{ \dot H^{1/2} ( \Ga)} $$
\end{theorem} 

{\bf Remark.} For domains in Euclidean space and higher oder elliptic operators results of similar type see in [LM1].  

\smallskip  {\it Proof.} Let $ \xi_0,\xi_1,\xi_2, \dots $ be the sequence of orthonormal eigenvalues of the Laplacian $-\De_X$ on torus $M_a$  
with the eigenvalues $0 =\nu_0 <  \nu_1 \leq \nu_2 \leq \dots $ and $ \phi_0, \phi_1, \phi_2, \dots $ be the sequence of orthonormal eigenvalues of the Laplacian $ -\De_{S^2}$ on the sphere 
$S^2$ with the standard round metric and with the eigenvalues $0 < \mu_0 \leq \mu_1 \leq \mu_2 \leq \dots $. Then functions

\beq \label{SN2.1}  \psi_{ij} = \xi_i(x) \phi_j (y) \eeq
where $i,j = 0,1,2, \dots $, are  eigenfunctions  of $\De $ on $\Ga $ with the eigenvalues $\la_{ij} = \nu_i + \mu_j $.

Define extensions of functions $ \psi_{ij} $ to $G' $ by

$$ \Psi_{ij} (x,y) = \psi_{ij} (x,y/ |y|)I_{\sqrt {\mu_j} }( \sqrt {\nu_i} |y| ) / | I_{\sqrt {\mu_j} }( \sqrt {\nu_i}  ) |,$$ 
where $I_{\sqrt {\mu_j} }$ is the modified Bessel function of the first kind. A simple estimate for Bessel functions  at point $t=1$,

$$ I'_{\sqrt {\mu_j} }( \sqrt {\nu_i } t ) \leq  C  \sqrt {\nu_i + \mu_j } I_{\sqrt {\mu_j} }( \sqrt {\nu_i } t )$$ 
Then functions $\Psi_{ij} $ are harmonic in $G'$ and they form an orthogonal system in $L^2(G') $.

Let $\{ij\} \neq \{nk\}$. Then

$$0 =  \int_{G'} ( \Psi_{ij} , \De \Psi_{nk}) dxdy = - \int_{G'} ( \nabla \Psi_{ij} , \nabla \Psi_{nk}) dxdy + \int_{\Ga} ( \Psi_{ij} ,\frac \pa {\pa y}  \Psi_{nk}) d\si$$
$$ = - \int_{G'} ( \nabla \Psi_{ij} , \nabla \Psi_{nk}) dxdy $$
Thus functions $\Psi_{ij} $ form an orthogonal system in  $ H^1(G') $.
Therefore

$$ \int_{G'} |\nabla \Psi_{ij} |^2 dxdy \leq C \sqrt {\nu_i + \mu_j } $$
and if $\Psi_{ij} $ is an odd function of $y$

$$C \int_{G'} |\nabla \Psi_{ij} |^2 dxdy \geq  \sqrt {\nu_i + \mu_j } $$

Taking the expansion of $f$ in the basis $\psi_{ij}$ we get the first and the third  inequalities of the theorem.

Since the integral

$$ \int_{M_a} | \Psi_{ij}(x,y) |^2 dx $$
is a subharmonic function on $Y$ then by mean value theorem for subharmonic functions from the first inequality of Theorem \ref{3.1}
follows the second inequality.

\begin{theorem} \label{3.2} There is a constant $C$ such that if $f$ is a smooth function on $\pa G$ with a compact support then
 \begin{equation}\label{in3.2} \int_0^\infty \int_{ \ga_t} t | \nabla V |^2 d\si dt \leq C \int_{\pa G} | f|^2 d\si \eeq
where $d\si $ is an element  of the surface measure.
\end{theorem} 

\smallskip  Theorem \ref{3.2} follows from Dahlberg's  generalization of Lusin's area integral Theorem, [D], see  [M]. The direct proof of Theorem \ref{3.2} 
was given by D. Jerison and C.E. Kenig, [JK]. 

\smallskip {\bf Remark.}  Inequality  \eqref{in3.2} to  domain $ G' = M_a \times B_1$ one can be simply reduced to Theorem \ref{3.2}.

\medskip  A linear differential operator $\mathcal P$ defined in a domain $\Om \ss \R^n$ with smooth coefficients is called hypoelliptic if any solution
$u$ of the equation $\mathcal P u =f$ with a smooth function $f$ in $\Om$ is smooth in $\Om$. We will consider the second order hypoelliptic
operators. By Hörmander's theorem, [H1]  , for a second order hypoelliptic operator with a real principal part, the principal part must be a semi-definite quadratic form.
Except of the uniformly elliptic operators there are two important classes of the second oder hypoelliptic operators: the parabolic operators and the
Hörmander operators. Let $X_0,X_1, \dots , X_m$  be smooth vector fields defined in $\Om$. We will consider $X_i$ as the first oder differential operator,
namely a partial derivative in the direction $X_i$. Define

 \beq \label{Hor}  L = \sum^m_{i=1} X^2_i +X_0 \eeq
Operator $L$ called  Hörmander's operator if the rank of Lie algebra generated by $X_i$ at the tangent space $T_p\Om$ at each point $p\in \Om$ is equal to $n$.
Such operators by a theorem of Hörmander are hypoelliptic, [H1]. Different proofs of this fact were given by Oleinik and Radkevich, [OR], and by Kohn, [Ko].
The probabilistic interpretation of the results see in [KS].  

Following, [RS], we define the Sobolev norms induced by the vector fields $X_i$. Let $X_{j_1} \dots  X_{j_l} $ be a monomial with $0\leq j_s \leq m, \,
s=1,\dots , l$. We shall say that this monomial has weight $r$ if $r=r_1+2r_2$, where $r_1$ is the number of $X_j$'s that enter with $j$ between $1$ and $m$,
and $r_2$ is the number of $X_0$'s.

Now let $k$ be an integer, $1<p< \infty$. Denote by $S^p_k( \Om )$ the set of all $f\in L^p(\Om )$ such that

$$X_{j_1} \dots  X_{j_l} f \in L^p( \Om) $$ 
for all monomials of weight $\leq k$. For the norm we take

$$ || f ||_{S^p_k} = \sum || X_{j_1} \dots  X_{j_l} f ||_{L^p} , $$
where the sum is taken over all odered monomials of weight $\leq k$.

The following theorem of Rothschild and Stein, [RS], (see also [B]) plays  an important role in the proof of Theorem \ref{NN}.

\begin{theorem} \label{RS} Let $L$ be  Hörmander's operator \eqref{Hor}.  Let $\eta_1, \eta_2 \in C_0^\infty (\Om) $ are two cutoff functions such that $0\leq \eta_1 \leq \eta_2 \leq 1, \, \eta_2 =1$
on $\supp \,  \eta_1$. Then 

$$ || \eta_1 u ||_{S^p_{k+2}} \leq C \left( || \eta_2 Lu ||_{S^p_k} + || \eta_2 u||_{L^p} \right) ,$$ 
$k\geq 0, \, 1< p <\infty $.
\end{theorem} 




\medskip   Now we go to the heat equation on $\R^{n+1}$. Denote

$$ P = { \partial \over \partial t} - \De $$
where $\De$ is the Laplace operator defined on $\R^n$. 

Two theorems below give us optimal estimates in the Sobolev spaces for solutions of the heat equations. 

\begin{theorem} \label{LSU} For every $n\in \N$ and every pair of numbers $p,q>1$ there exists a constant $C=C(n,p,q) >0$ such that every 
compactly supported smooth function $u\in C_0^\infty ( \R^{n+1})$  satisfies the estimate
$$ \int_{-\infty}^\infty || \pa u / \pa t ||^q_{L^p(\R^n )} dt  \leq  C  \int_{-\infty}^\infty || \pa u / \pa t - \De u ||^q_{L^p(\R^n)} dt  $$
\end{theorem} 

\smallskip For $p=q$ the theorem was proved by Ladyshenskaya-Solonikov-Uralćeva, [LSU]. For all $p,q>1$ see, [HP]  

\begin{theorem} \label{S} For every $n\in \N$ and $p\geq 2$ there exists a constant $C=C(n,p)$ such that every compactly supported smooth function 
 $u\in C_0^\infty ( \R^{n+1})$  satisfies the estimate
 
 $$ || \nabla u (\cdot , T) ||_{L^p(\R^n) }   \leq  C \left( \int_{-\infty}^T || \pa u / \pa t - \De u ||^2_{L^p(\R^n)} dt \right)^{1/2} $$
 where $\nabla = \nabla_x$.
\end{theorem}

\smallskip Theorem \ref{S} was proved in the preprint of Salamon, [S]. For  convenience of the reader we give a simple proof of this theorem. The initial step of the proof is similar to  one done in [S].  Notice,
that for $p=2$ the result is well known, [LM2].

\smallskip  {\em Proof.} Let

\beq \label{Peq1}  {\pa u \over \pa t } - \De u =f \eeq
Set

$$ v =  {\pa u \over \pa x_i } , $$
$i=1,\dots , n$. Then

$$ {\pa v(x,t) \over \pa t } - \De v(x,t) = {\pa f(x,t) \over \pa x_i } $$
Let $2 \leq p $. Multiplying the last equality on $|v|^{p-1} $ and integrating over $\R^n$ we get

$$ \frac1p \int_{\R^n}  {\pa |v|^p \over \pa t } dx + (p-1) \int_{\R^n} |v|^{p-2} |\nabla v |^2 dx +  \int_{\R^n} f {\pa |v|^{p-1} \over \pa x_i } dx  =0$$
Integrating the last equality from $-\infty $ to $T$ we get

\beq \label{Peq2}  || |v( \cdot , T)|^p ||_{L^1(\R^n)} \leq p(p-1) \int^T_{-\infty} \int_{\R^n} (|f(x , t)|+ |{\pa u \over \pa t}|) |v(x, t) |^{p-2} |\nabla v | dxdt \eeq

By Theorem \ref{LSU}  from the equation \eqref{Peq1}

$$ ||  \nabla v||_{L^2(-\infty ,T) L^p (\R^n)}  \leq C || f ||_{L^2(-\infty ,T) L^p (\R^n)}       $$
and
$$ ||  \pa u / \pa t||_{L^2(-\infty ,T) L^p (\R^n)}  \leq C || f ||_{L^2(-\infty ,T) L^p (\R^n)}       $$
By ``Mixed Hölder inequality", \eqref{Peq2} and the last two inequalities

$$ \int^T_{-\infty} \int_{\R^n} (|f(x , t)|+ |{\pa u \over \pa t}|) |v(x, t) |^{p-2} |\nabla v | dxdt  $$
$$ \leq   C || |v|^{p-2} ||_{ L^\infty (- \infty ,T)  L^{p/(p-2)} (\R^n) } || f ||^2_{L^2(-\infty ,T) L^p (\R^n)}       $$
Thus

$$ || |v( \cdot , T)|^p ||_{L^1(\R^n)} \leq  C || |v|^{p-2} ||_{ L^\infty (- \infty ,T)  L^{p/(p-2) }(\R^n) }|| f ||^2_{L^2(-\infty ,T) L^p (\R^n)}     ,  $$

$$ || v( \cdot , T) ||_{L^p(\R^n)} \leq  C  || v ||^{(p-2)/p}_{ L^\infty (- \infty ,T)  L^{p }(\R^n) } || f ||^{2/p}_{L^2(-\infty ,T) L^p (\R^n)}       $$
Since the right hand side of the desirable inequality is monotone non-decreasing function, we may assume without loss that for all $t<T$

$$ || v( \cdot , t) ||_{L^p(\R^n)} \leq  || v( \cdot , T) ||_{L^p(\R^n)} $$
Then
$$ || v( \cdot , T) ||^{2/p}_{L^p(\R^n)} \leq  C  || f ||^{2/p}_{L^2(-\infty ,T) L^p (\R^n)}       $$
and the theorem follows.

.

\section{Holomorphic solenoidal vector fields} 
Let $a,s,r >0$. Denote $M_s = \R^3 / s \Z^3$,

$$ \Omega^s_r =   \{ (x+iy) \in \C^3:  x\in M_s , y\in B_r\}$$
$$ \Omega = \Omega^a_1, \;\;\; \Om_r=\Om^a_r   $$

Denote by $X$ and $Y$ correspondingly the real and the imaginary subspaces of $\C^3$: $\Re Z = X=\{x_1,x_2,x_3\} ,\; \Im Z = Y= \{ y_1,y_2,y_3 \} $.

Let  

$$W(z)=(w_1{\pa \over \pa z_1},w_2{\pa \over \pa z_2},w_3{\pa \over \pa z_3})$$  
$w_i=u_i+\sqrt{-1} v_i$, be a holomorphic solenoidal vector field defined in $\C^3$, i.e., it  satisfies

$${ \partial w_i\over  \pa \bar z_i }=0   $$

\beq \label{divw0}\sum_{i=1}^3 {\partial w_i\over  \pa  z_i }=0   \eeq

Notice, that since $w_i$ are holomorphic functions then the equality \eqref{divw0}  on the real subspace of $\C^3$ implies it on the whole domain of holomorphy of $w$.

We assume that the vector field $W$ is periodic in $X$ with the periods $a\Z $, or equivalently, that $W$ is defined on $\Omega$.

We assume that $W$ is real on the real subspace of $\C^3$:

Define real vector fields

$$U=(u_1,u_2,u_3) , \; \; V=(v_1,v_2,v_3) , $$ 

\beq \label{Real}  V=0  \;\;\; on \;\; \; X \eeq

Vector fields $U$ and $V$ are divergence free on every totally real three-dimensional affine plane $L_y \ss \C^3$ :

$$ L_y= \{ (x+iy) : x\in X \} ,$$
$$ \div U = \sum {\partial u_i\over  \pa  x_i }= \div V = \sum {\partial v_i\over  \pa  x_i } = 0 \;\;\; on \;\;\;  L_y $$

Denote

$$\tilde w = (\Im Z, W)=  \tilde u + \sqrt{-1} \tilde v $$

 Note that $\tilde w $ is not a holomorphic function.

\begin{lemma}\label{BasicHol} Let $a \geq1$. 

$$ || \tilde v | V |^2 ||_{L^1 ( \pa \Om ) } \leq C || V||^2_{H^{1/2}( \pa \Om) } || V||_{L^2( \pa \Om) } $$
\end{lemma}

\smallskip  {\em Proof.} Denote by $H_p, \, p\in \pa \Om$, a complex holomorphic subbundle $T^{1,0} \pa \Om$ of the tangent bundle to $\pa \Om$.  We will also consider
$H$ as a real 4-dimensional section of the tangent bundle $T \pa \Om$.

Let $g$ be a Riemannian metric on $\pa \Om$ induced by the imbedding $\pa \Om \ss \C^3 $. Denote by $\De = div \, \nabla $ the Beltrami-Laplace operator on $(\pa \Om , g)$, where $div$ and $ \nabla $ 
are taken at the Riemannian metric $g$. Let $p\in \pa \Om , \, p=( x,y), \, x\in X , \, y\in S^2 $, where $S^2 \in Y$ is a unit two-dimensional sphere centered at the origin.
Then $T_p \pa \Om = X \oplus \Sigma_y $, where $\Sigma_y $ is a two-dimensional plane tangent to $S^2$ at $y$ and orthogonal to $X$. Therefore

$$ \nabla_{\pa \Om } = \nabla_X + \nabla_S  $$
and hence $\De $ can be written as 
 
 $$ \De = \De_X + \De_{S^2} ,$$
 where $\De_X$ is the Laplace operator on the subspace $X$ and $\De_{S^2}$ is the Beltrami-Laplace operator on $S^2$ with a standard round metric.

 For any smooth function $f$ on $\pa \Om$ we have
 
 \beq \label{sq1}  \De_X f^2 = 2f \De_X f + 2 |\nabla_X f|^2 , \eeq 
 
 \beq \label{sq2}  \De_{S^2} f^2 = 2f \De_{S^2} f + 2 |\nabla_{S^2} f|^2 \eeq

Since the surface $\pa \Om$ is smoothly embedded in $\C^3 \simeq \R^6$ we have a natural injective linear map

$$ i_x: T\pa\Om \to T \C^3 $$

Denote by $\La_p$ the Laplacian on the affine 2-complex-dimensional plane $H^*_p =i_* H_p$:

$$ \La =\La_p = 4 \sum_{i=1}^2 { \pa^2 \over \pa z^p_i \pa \bar z^p_i }  $$
where $z^p_1,\, z^p_2$ is an orthonormal basis on $H^*_p$. Then for any smooth, defined in a neighbourhood 
of $\pa \Om$ function $f$

 \beq \label{La}  \La f(p) = \De f(p) +2{\pa f \over \pa r } (p) - L^2 f(p) \eeq
 where $r=|y|, \, L \in T\pa \Om , \, L \perp H_p , \, | L | =1 $.
 
 Notice that if $f$ is a smooth function on $M$ and $e\in \R^3$ is a fixed vector then
 
$$ \int_M {\pa f \over \pa e} dx =0 $$
and hence  

 \beq \label{fe2}  \int_M {\pa^2 f \over \pa e^2} dx =0  \eeq

 Since for any pluriharmonic function $h$ in $\bar \Om , \; \La h =0$, 
  and
  
  $$ \int_{\pa \Om} h {\pa h \over \pa n} d\si = \int_\Om | \nabla h |^2 dxdy $$
   Then from  \eqref{sq1} - \eqref{fe2} and from self-adjointness  of the Beltrami-Laplace operator $\De $ on $\pa \Om$, it follows   
  
  $$ 0= \int_{\pa \Om}  \De  h^2 d \si    =  \int_{\pa \Om} ( \La h^2 - 2  {\pa h^2 \over \pa n } + L^2 h^2 ) d\si $$
  $$  = 2 \int_{\pa \Om} ( - 2 h {\pa h \over \pa n } + |\nabla_{H}h |^2)d\si 
   +  \int_{S^2} \int_{M} L^2 h^2 dxd\si_0   $$
 $$ = 2\int_{\pa \Om } |\nabla_{H} h |^2d\si -4 \int_\Om  |\nabla h |^2 dxdy  $$
  where $n$ is a normal vector to $\pa \Om $, $ \nabla_H = \nabla_{H_p}, \, p\in \pa \Om$, be the gradient along $H_p$, $\si$ and $\si_0$ are area elements on $\pa \Om$ and $S^2$ correspondingly. Thus by Theorem \ref{3.1}
  
  \beq \label{in.h}  \int_{\pa \Om } |\nabla_{H} h|^2d\si  =  2\int_\Om |\nabla h|^2 dxdy\sim || h ||^2_{ H^{1/2} ( \pa \Om )}  \eeq
  
  
    From inequality \eqref{in.h}  follows
    
  \beq \label{in.V} \int_{\pa \Om } |\nabla_{H} V|^2d\si  \lesssim || V ||^2_{H^{1/2} ( \pa \Om )}  \eeq    
   The same inequality holds for $U$. Thus from the equality
 
  \beq \label{N.8.3} \int_{\Om } |\nabla V|^2d\si  = \int_{\Om } |\nabla U|^2d\si  \eeq
  follows
 
 \beq \label{in.W} \int_{\pa \Om } |\nabla_{H} W|^2d\si  \lesssim || V ||^2_{ H^{1/2} ( \pa \Om )}  \eeq    
 
Let $p=(x,y) \in \pa \Om$ and let $z'_1(p)= y_1,z'_2(p)=x'_2(p)+iy'_2(p), z'_3(p)=x'_3(p)+iy'_3(p)$ be an orthonormal coordinate system in $\C^3$, such that $span \{ z'_2,z'_3\} = H_p$. We write the field $W$ in the coordinates $z'$ as 

$$W=(w'_1(z'){\pa \over \pa z'_1},w'_2(z'){\pa \over \pa z'_2},w'_3(z'){\pa \over \pa z'_3})$$  
Then  for  $p\in \pa \Om$
 
 $$\tilde w(p) = w'_1(p) $$
 
 Let $ \frak J = \{  \pa w_i / \pa z_j  \} $ be the jacobian matrix of field $w$. Then
 
 $$ \tr \frak J = \sum_{i=1}^3 {\partial w_i\over  \pa  z_i }=0    $$
  Let $\mathcal U : (z'_1,z'_2,z'_3) \to (z_1,z_2,z_3) $ be the unitary transformation of $\C^3$. From similarity invariance of the trace, $ \tr \; \mathcal U \, \frak J \, \mathcal U^* =0$
  and therefore 
  
  $$  \sum_{i=1}^3 {\partial w'_i\over  \pa  z'_i }=0    $$

 Denote $w'_i  = u'_i + \sqrt { -1 } v'_i$,
 
$$U'=(u'_1,u'_2,u'_3) , \; \; V'=(v'_1,v'_2,v'_3) $$ 
Then 

 \beq \label{divUV} \div U' =\div V' =0 \;\;\; on \;\;\; Y \eeq    
From \eqref{in.W}  we get

 \beq \label{in.ti} \int_{\pa \Om } |\nabla_{H}\tilde w|^2d\si  \lesssim || V ||^2_{ H^{1/2} ( \pa \Om )}  \eeq    
 


Since  $\La V =0$ we get    from the  Cauchy-Riemann equations
 
$$  \La \tilde v (y)  =   - 2{ \pa \tilde v \over \pa r }(y) = 2{ \pa \tilde u \over \pa x_1 }(y)  , $$



\smallskip  We introduce now two suitable partitions of unity on $\pa \Om $.

1. Let $x_1,\dots ,x_k \in M$ be $k$ vertices of the lattice $\frac a{[a]} \Z^3$ in $M_a, \,
k \leq [a]^3,$ and let
$y_1, \dots ,y_8 \in S^2$ be vertices of a cube inscribed in $S^2$. For each pair $i,j, \; 1\leq i \leq k,\, 1\leq j \leq 8$ denote

$$\tilde G_{i,j} = B_{a/[a]}^{x_i} \times \mathcal B_1^{y_j} ,$$
$$G_{i,j} = B_{3a/2[a]}^{x_i} \times \mathcal B_{3/2}^{y_j} ,$$
$$ \widehat G_{ij} = B_{2a/[a]}^{x_i} \times \mathcal B_2^{y_j} ,$$
where $B_r^x \ss M_s$ be a ball of radius $r$ centered at $x$ and $\mathcal B_r^y \ss S^2$ be a geodesic  disk on $S^2$ of radius $r$ centered at $y$.

We have

$$\pa \Om \ss \cup_{i,j}\tilde G_{ij} $$
and the {\it multiplicity of covering} of $\pa \Om$ by $\{ \widehat G_{ij} \}$ is uniformly bounded: for any pair $\{ i,j\}$

$$ \sharp \{ \{ k,n\} : \widehat G_{ij} \cap \widehat G_{kn} \neq \emptyset \} \leq 208 $$

Define $\phi_{ij}^0 \in C_0^\infty( G_{ij} )$, such that $0\leq \phi_{ij}^0 \leq 1$,
$$\phi_{ij}^0 =1 \;\;\; on \;\;\; \tilde G_{ij} $$
and define a partition of unity $\phi_{ij}$ on $\pa \Om$ by

$$\phi_{ij} = \phi_{ij}^0 / \sum_{k,n} \phi^0_{kn} $$

Since $a\geq 1$ we can choose $\phi_{ij} $ with uniformly bounded derivatives. There is a constant $K$ independent of $a$, such that

$$ || \phi_{ij} ||_{C^2 ( G_{ij)}} \leq K $$ 

For any $a\geq 1$ the rank of the partition $\phi_{i,j}$ is uniformly bounded  $\leq 208$.

The domains $ \widehat G_{ij} $ are mutually  congruent.

\smallskip We transfer the  partition $\phi_{ij} $ from $\pa \Om $ to $\pa \Om_r, \, r<1$, setting  $\phi^r_{ij} (x,y) = \phi_{ij} ( x, y/r) $. Correspondingly,
 domains $ G_{ij}, \,  \widehat G_{ij} $  moved from $\pa \Om $ to $\pa \Om_r, \, r<1$,  we denote  $ G^r_{ij}, \,  \widehat G^r_{ij} $ 

\smallskip 2. Let $x_1,\dots ,x_{8} \in M$ be $8$ vertices of the lattice $\frac a2 \Z^3$ in $M$, and  $y_1, \dots , y_8 \in S^2$ be  vertices defined in the previous section. Denote

$$G'_{i,j} = B_{a/2}^{x_i} \times \mathcal B_1^{y_j} ,$$
$$ \widehat G'_{ij} = B_{a}^{x_i} \times \mathcal B_2^{y_j} $$

We have

$$\pa \Om \ss \cup_{i,j} G'_{ij} $$

Define $\psi_{ij}^0 \in C_0^\infty(\widehat G'_{ij} )$, such that $0\leq \psi_{ij}^0 \leq 1$,
$$\psi_{ij}^0 =1 \;\;\; on \;\;\; G'_{ij} $$
and define a partition of unity $\psi_{ij}$ on $\pa \Om$ by

$$\psi_{ij} = \psi_{ij}^0 / \sum_{k,l} \psi^0_{kl} $$



\smallskip Let $B_1 \ss \R^2 = Y^* = \{  y^*_1,y^*_2\},\, B_s\ss \R^3 = X = \{ x_1,x_2,x_3\}$ be the disk and the ball of radii $1$ and $s$ correspondingly. Let 

$$\xi : \mathcal B_2^{y_j} \longrightarrow B_1 $$
be a diffeomorphism and

$$ \iota : B_a^{x_i} \longrightarrow B_a ,$$
$$\iota (x') =x'-x_i $$
be an isometry

Define  diffeomorphisms

$$\Phi_{ij} : \widehat G'_{ij} \longrightarrow B_1\times B_a ,$$
$$ \Phi_{ij} (x,y) = ( \iota (x), \xi (y)) $$
 
\smallskip Below we reduce  the estimates of vector fields $W,\,  \tilde w$ on $\pa \Om$ to the sum of corresponding estimates the products $\phi_{ij}W, \, \phi_{ij} \tilde w $ we on the domains $G_{ij}$.  We fix now one of these domains $G_{ij}$ and denote it by $G$, domain $\widehat G_{ij}$ denote by $\widehat G$, the ball $\mathcal B_2^{y_j}$ denote $\mathcal B$,
function $\phi_{ij}$. Denote  functions $\psi_{ij}, \Psi_{ij}, \Phi_{ij}$ from the partition 2  by $\psi, \Psi, \Phi$ correspondingly

\smallskip  Let $X_1(y), \, X_2(y), \, y \in \mathcal B$ be two one-real-dimensional smooth sections of the tangent bundle $T \mathcal B \ss T S^2$, such that $X_1(y), \, X_2(y)$ form
an orthonormal frame on $T_y S^2$. Denote $X_3= JX_1,\, X_4=JX_2$, where $J$ is a complex structure on $\C^3$.  Then for any $p\in G,\; X_1,\dots , X_4$ is an orthonormal frame on $H_p$.

Define on $G$ a second order elliptic operator of the Hörmander type:
  
 $$ \mathcal H = \sum_{i=1}^4 X_i^2 $$ 
 
 Since the Levi form  $ \mathcal L $ of $\pa \Om$, 
 
$$  i\pa \bar \pa |y|^2 $$
  is non-degenerate, the Frobenius form $ \mathcal F (X,X') = - \mathcal L (X, JX') = [X,X'] \, \mod  \, H $ is also non-degenerate. Hence, 
 the vectors $X_i$ satisfy the Hörmander's condition of hypoellipticity  of $\mathcal H $: the $span $ of $X_i$ and their commutators up to the finite order span $T_pG $ at each point $p\in G$. There is a simple way to show non-vanishing of the above Levi form. Assume, in general case, that
  $D\ss \C^3$ is a smooth domain and $s$ is a smooth in $\C^3$ function such that $D$ is a nodal domain of $s$ and $\nabla s \neq 0$ on $\pa D$. It is well
  known that $\pa D$ is Levi flat if and only if the determinant
  
  $$  \det \begin{pmatrix} 0 & s_{z_1} & s_{z_2} & s_{z_3} \\ s_{\bar z_1} & s_{\bar z_1 z_1} & s_{\bar z_1 z_2} & s_{\bar z_1 z_3}  \\
  s_{\bar z_2} & s_{\bar z_2 z_1} & s_{\bar z_2 z_2} & s_{\bar z_2 z_3}  \\ s_{\bar z_3} & s_{\bar z_3 z_1} & s_{\bar z_3 z_2} & s_{\bar z_3 z_3}    \end{pmatrix}  $$
  is equal to zero, [H2], [ML].  In our case: $ D=\Om, \, s= 1 - |y|^2$ and the last determinant is equal to $1$ (this computation is simple in the basis
  $z'_1, z'_2, z'_3$).  Other way,  one can get non-degeneracy of  $\mathcal F$   as a consequence of an old 
  result of Sommer, [So] (see also [F]). 
  
 Alternatively it's  possible to check directly that $X_1, \dots , X_4, \, [X_1,X_3] $ span $T_pG$.
  
  Define 
  
  $$ \mathfrak L = \De -L^2 $$

The principal part of $\mathfrak L$ coincide with the principal part of $ \mathcal H$. 

From equality \eqref{La} since $V$ is a pluriharmonic function we get

 \beq \label{10.5} \mathfrak L V= -2 { \pa V \over \pa r } \eeq
Thus from the equality

$$ \mathfrak L  \Im z = 2$$
follows
  
 \beq \label{Ltilde} \mathfrak L \tilde v (y) = - 4 { \pa \tilde v \over \pa r }(y) + 2  \tilde v (y)  \eeq
For $p=(x,y) \in \pa \Om $

$$ { \pa \tilde v \over \pa r }(p) = { \pa v'_1 \over \pa y_1 }(p) $$
Hence, from \eqref{divUV}  

 \beq \label{10.8} { \pa \tilde v \over \pa r }(p) = - { \pa v'_2 \over \pa y_2 }(p) - { \pa v'_3 \over \pa y_3 }(p) \eeq
Therefore
  
 \beq \label{frak} |\mathfrak L \tilde v | \leq 4 ( | \nabla_H V | + |V| )\eeq
 
  
 
 From inequality \eqref{frak} follows

$$ || \phi \mathfrak L  \tilde v||_{L^2(G) }\leq C  ( || \nabla_H V ||_{L^2( G)} + || V ||_{L^2( G)} ) $$
Since $\phi \leq 1$, by Theorem \ref{RS} 

$$ || \phi \tilde v||_{S_2^2(G) }\leq C  ( || \nabla_H  V ||_{L^2( \widehat G)} + || V||_{L^2( \widehat G)} ) $$
Since

$$  || \phi \tilde v||_{S_2^2(G) } \geq  || \phi  \nabla^2_H \tilde v||_{L^2(G) } + || \phi  \nabla \tilde v||_{L^2(G) }    -2 || \nabla \phi  \nabla_H \tilde v||_{L^2(G) }    $$
$$ -  || (\nabla^2_H \phi ) \tilde v||_{L^2(G) } - || (\nabla_H \phi ) \tilde v||_{L^2(G) }    $$
$$ \geq  || \phi  \nabla^2_H \tilde v||_{L^2(G) } + || \phi  \nabla \tilde v||_{L^2(G) }    -C  ( || \nabla_H  V ||_{L^2( \widehat G)} + || V||_{L^2( \widehat G)} ) $$
we get

$$ || \phi (| \nabla^2_H \tilde v|  + | \nabla \tilde v| )||_{L^2(G) }
  \leq C   ( || \nabla_H V ||_{L^2( \widehat G)} + || V ||_{L^2( \widehat  G)} ) $$
and hence
$$ || \phi (| \nabla^2_H \tilde v|  + | \nabla \tilde v| )||^2_{L^2(G) } \leq  C   ( || \nabla_H V ||^2_{L^2( \widehat G)} + || V ||^2_{L^2( \widehat  G)} ) $$

Since the multiplicity of the covering of $\pa \Om$ by the domains $ \widehat  G_{ij} $ is uniformly bounded by $208$, taking the sum of  the last inequalities over all domains $G_{ij} $ from Proposition \ref{PU} 

$$  ||  \tilde v||_{S_2^2(\pa \Om )}\leq  C   ( || \nabla_H V ||^2_{L^2( \pa \Om)} + || V ||^2_{L^2( \pa \Om)} ) $$
and from  \eqref{in.W}  we get

 \beq \label{i.tiV}  ||  \tilde v||_{S_2^2(\pa \Om )}\lesssim    ||  V ||_{H^{1/2}( \pa \Om)}  \eeq

  From the inequality  \eqref{in.h} we have
  
    \beq \label{10.12} ||  \nabla_H V ||_{ L^2 (\pa \Om  ) } \lesssim ||  V ||_{ H^{1/2} \pa \Om )} \eeq

From inequality \eqref{i.tiV}

 \beq \label{L5}  ||  \tilde v||_{H^1(\pa \Om )}\lesssim    ||  V ||_{H^{1/2}( \pa \Om)}  \eeq

 Since $\tilde v$ is a harmonic function in $\Om $  by Theorem \ref{3.2} we have
 
 \beq \label{h3} \int_0^1(1-r) ||\nabla V||^2_{L^2 ( \pa \Om_r ) } dr \lesssim   || V||^2_{L^2( \pa \Om )}  \eeq
 since 
 
  $$  \int_0^1(1-r) ||\nabla V ||^2_{L^2 ( \pa \Om_r ) } dr = \int_0^1 \int_{\Om_r} |\nabla V |^2 dxd.y  $$
we get

  \beq \label{NV1}  \int_0^1|| V||^2_{H^{1/2}  ( \pa \Om_r ) } dr \lesssim   || V||^2_{L^2( \pa \Om )} \eeq
  From the last inequality and inequality \eqref{L5} 
  
  $$  \int_0^1 ||\nabla \tilde v||^2_{L^2 ( \pa \Om_r ) } dr \lesssim   || V||^2_{L^2( \pa \Om )} $$
  Hence

 \beq \label{L4}  ||  \tilde v||_{H^{1/2}(\pa \Om )}\lesssim    ||  V ||_{L^2( \pa \Om)}  \eeq

\smallskip  Denote on $B_1\times B_a \ss \R^5 = X \times Y^*, \,Y^* = \{ y_1^*,y_2^* \}$
  functions
  
$$ V^* = \psi V ( \Phi^{-1}) , \;\;\;  \tilde v^* = \psi \tilde v ( \Phi^{-1}) $$

  Let $y'\in Y^* $. Denote by $ X^*( y') $ two-dimensional subspace $x^*_1, x^*_2$ in $X$ orthogonal to $J\Phi^{-1} ( y')$ and by $T ( y') \ss X $ an axis (one-dimensional subspace) parallel to the vector  $J\Phi^{-1} ( y')$ .
Thus $X= X^*( y') \times  T( y')$ and $\R^5$ is given as a product of orthogonal subspaces $Y^*, X^*, T$. Then 


 \beq \label{nnnn} || V^* ||_{ L^2(T) L^2 (X^*) L^2(Y^*)}  \lesssim || V||_{L^2( \pa \Om) }  ,\eeq

Inequality \eqref{10.12}  implies

$$ || V^* ||_{ L^2(T) L^2 (X^*) H^1(Y^*)}  \lesssim    || V||_{H^{1/2}( \pa \Om) }  ,$$

$$  || V^* ||_{ L^2(T) H^1 (X^*) L^2(Y^*)}  \lesssim    || V||_{H^{1/2}( \pa \Om) }  $$
 From the two last inequalities

$$ || V^* ||_{ L^2(T) H^{1/2} (X^*) H^{1/2}(Y^*)}  \lesssim    || V||_{H^{1/2}( \pa \Om) } $$
 Sobolev embedding theorems   imply

 \beq \label{SN10} || V^* ||_{ L^2(T) L^4 (X^*) L^4(Y^*)}  \lesssim    || V||_{H^{1/2}( \pa \Om) }  \eeq

\medskip Define the heat operators,

$$ P_1  = { \pa \over \pa t} - { \pa^2 \over \pa (x^*_1)^2 } - { \pa^2 \over \pa (x^*_2)^2 } ,$$

$$ P_2  = { \pa \over \pa t} - { \pa^2 \over \pa (y^*_1)^2 } - { \pa^2 \over \pa (y^*_2)^2 } $$

From \eqref{i.tiV} follows 

     $$ || P_1 \tilde v^* ||_{L^2(Y^*) L^{2}(X) }  \lesssim  || V||_{H^{1/2}( \pa \Om) } ,$$
     
          $$ || P_2 \tilde v^* ||_{L^2(Y^*) L^{2}(X) }  \lesssim  || V||_{H^{1/2}( \pa \Om) } $$

     Thus by Theorem \ref{S} 
     
     $$ || \tilde v^* ||_{L^2(Y^*) H^1(X^*) L^\infty (T) }  \lesssim  || V||_{H^{1/2}( \pa \Om) } ,$$
     
          $$ || \tilde v^* ||_{L^2 (X^*) H^1 (Y^*)  L^\infty (T) }  \lesssim  || V||_{H^{1/2}( \pa \Om) } $$
From the two last inequalities     
     and by Sobolev embedding theorems 
  
     $$ || \tilde v^* ||_{L^4(Y^*) L^{4}(X^*) L^\infty (T) }  \lesssim    || V||_{H^{1/2}( \pa \Om) } $$
     Thus from the last inequality and \eqref{SN10}

$$ || \tilde v^* |V^*| ||_{ L^{2 }(X \times Y^*)}  \lesssim  || V||^2_{H^{1/2}( \pa \Om) } $$
a from the inequality \eqref{nnnn}

$$ || \tilde v^* |V^*|^2 ||_{ L^{1 }(X \times Y^*)}  \lesssim  || V||^2_{H^{1/2}( \pa \Om) } || V||_{L^2( \pa \Om) } $$

Since $\sharp \{ \psi_{ij} \} = 64$ it follows

 \beq\label{SN11} || \tilde v | V |^2 ||_{L^1 ( \pa \Om ) } \leq C || V||^2_{H^{1/2}( \pa \Om) } || V||_{L^2( \pa \Om) } \eeq

The lemma is proved.

\smallskip {\bf Remark.} For the proof of Theorem \ref{NN} we are requiring  an inequality of the following type

 \beq\label{SN12} || \tilde v | V |^2 ||_{L^1 ( \pa \Om ) } \lesssim  || V ||^{\al }_{H^{1/2} ( \pa \Om )} || V ||^{3 - \al } _{L^2( \pa \Om )} \eeq
with a positive constant $\al \leq 2$. 
The divergence-free assumption for vector field $W$ is essential for  inequalities of such type.
The better outcome for the solenoidal fields provided by the regularity properties of  term $\tilde v$. For solenoidal fields the a priori 
estimate of $\tilde v$ in $H^1(\pa \Om) $ is a consequence of \eqref{i.tiV}.  Such estimate does not hold in general for pluriharmonic  in $\Om$ functions. Boundary estimates for 
holomorphic functions in $\Om$ are closely related to the regularity property of the Hörmander operator $\frak L$. The
Gaussian bounds of the corresponding heat equation can be given in terms of Carnot-Carathéodory metric, generated by 
holomorphic subbundle of the tangent bundle  of $ \pa \Om$,  [JS], [BB], [KS], [FS].  In terms of the geodesic distance, the same
 estimates for the Schrödinger equation on Riemannian manifolds,  was obtained in [LY]. The  Carnot-Carathéodory metric is an anisotropic non-smooth metric,   [NSW], and as a consequence, there are no estimates of  type \eqref{i.tiV}  for pluriharmonic functions. On the other hand, the complete differential of function  $\tilde v$ on $\pa \Om$ can be written in terms  of Wirtinger derivatives of $W$  and that's compensate anisotropy of the metric.
 
 We didn't employ  inequality \eqref{L4} but we keep it because it gives a way for alternative proof of  inequality \eqref{SN11}.

\medskip 

 Since $V=0$ on $X$, $V$ is an odd function on $Y$. Hence for any $x\in X$

$$ \int_{B_1} V(x,y) dy =0 $$
Thus by Poincaré inequality

  \beq \label{Poin}   || V ||^2_{L^2 ( \Om ) } \leq C \int_\Om | \nabla V |^2 dxdy \eeq
Hence  by
Theorem  \ref{3.1} we can rewrite  Lemma   \ref{BasicHol} in the following form

\begin{lemma}\label{BasicHolP} Let $a \geq 1$. 

$$ || \tilde v |V|^2||_{L^1(\pa \Omega )}  \leq C || V||_{L^2( \pa \Om) }   \int_\Om | \nabla V |^2 dxdy  $$
\end{lemma}

Making homotetie  of $\Om_r $ to $\Om^{a/r}_1 , \, z \to z/r$ and taking into account the scaling factors of the Sobolev norms and scaling of $\tilde v$ we transfer Lemma  \ref{BasicHolP} 
to the domain $\Om_r$.

\begin{lemma}\label{BasicHol2} Let $a \geq1, \, r\leq 1$.  

$$ || \tilde v | V |^2 ||_{L^1 ( \pa \Om_r ) } \leq {C \over r^{1/2} } || V||_{L^2( \pa \Om_r) }   \int_{\Om_r} | \nabla V |^2 dxdy  $$
\end{lemma}

\section{Proof of Theorem \ref{NN} }

We start with some auxiliary results. Denote

$$ A = \{ (x,y)  \in \R^2, \; 0<x< \pi  /b, \, 0<y<r \},  $$

$$ a = \{ (x,y) \in \pa A, \; y= r \} $$

Let $u$ be a harmonic function in $A$, smooth in $\bar A$ and $u=0$ on $\pa A \setminus a $.

\begin{lemma}\label{SN-1} Let $r , b \leq 1$. Then
 \beq\label{SN2.9} \int_A | \nabla u |^2 dxdy \leq  \int_A | \pa u / \pa x|^2 dxdy + { 1 \over r } \int_a u^2 dxdy \eeq
\end{lemma}

{\it Proof.} Denote

$$ \phi_n  = \sin nx b \sinh ny b $$
By elementary computations harmonic function $\phi_n $ satisfies the inequality 

$$ \int_A | \nabla \phi_n |^2 dxdy \leq  \int_A | \pa \phi_n / \pa x|^2 dxdy + { 1 \over r } \int_a \phi_n^2 dxdy $$

 We expand $u$ into series

 \beq\label{SN2.10}  u = \sum_{n=1}^\infty a_n \phi_n \eeq
Since $\{ \phi_n \}$ are orthogonal functions in $ L^2(A), \, H^1(a) $ and in $ L^2(A) $  for the derivatives $ \pa \phi_n / \pa x $, inequality \eqref{SN2.9} follows for  function $u$.

 \medskip  

Denote 

$$ g(y) = \int_0^{\pi / b} u^2(x,y) dx , $$
from the expansion \eqref{SN2.10}  it follows that for $0\leq y \leq r , \, n \in \N $

  \beq\label{SN2.2}  g^{ (n) } (y) \geq 0 \eeq

 \medskip  

Let $f(z)$ be a holomorphic function in the strip $\Sigma = \{ z\in \C, \, \Im z<1\} $. Assumed that $f$ is uniformly bounded in $\Sigma $ and is
an almost periodic function on the real line $l = \{ (\xi , \eta): \, \eta=0\} $.  Then $f$ is an almost periodic function on any line $l_a = \{ (\xi ,\eta), \, \eta=a \}, \, -1<a<1 $, 
see [C]. Therefore, if $h$ is a harmonic, uniformly bounded function in $\Sigma $, $h=0$ on $l$ and $\pa h / \pa \eta $ is an almost periodic
function on $l$, then $\nabla h$ is an almost periodic function on any line $l_a, \, |a| < 1$.

Denote

$$ d^1(\xi ) = \int^r_0 | \nabla h( \xi , \eta)|^2d\eta ,$$

$$ d^2(\xi ) = \int^r_0 |  h_\xi ( \xi , \eta)|^2d\eta ,$$

$$ d^3(\xi ) =   h^2( \xi , \eta) $$
$\xi \in \R , \, r<1$.  Since 

$$ \left| \int^r_0 h(T,\eta) h_\xi( T,\xi) d\eta - \int^r_0 h(-T,\eta) h_\xi( -T,\xi) d\eta  \right| < C $$ 
where constant $C= C(r)>0$ is independent from $T$, it follows that

 \beq\label{SN2.4}   \mathcal M (d^1) = \mathcal M (h(\cdot , r) h_\eta (\cdot, r)) \eeq
where $\mathcal M (f) $ is the mean of an almost periodic function $f$,

$$  \mathcal M (f) = \lim_{T \to \infty } { 1\over 2T} \int^T_{-T} f d\xi $$

As a consequence of Lemma \ref{SN-1} we have

\begin{lemma}\label{SN-2} Let $\eta \leq 1$. Then
$$  \mathcal M (d^1) \leq   \mathcal M (d^2) + \frac 1\eta  \mathcal M (d^3) $$
\end{lemma}

{\it Proof.}  Since $u$ is  uniformly bounded in $\Sigma $ then for any $r<1, \, \nabla u$ is uniformly bounded in the strip

$$ \Sigma_r = \R \times [0, r] $$

Define a smooth function $f_T$ on $\R$ with the support $( - T - 1, T+1 ), \, T>0 $ such that $f=1$ on $(-T,T)$ and $f(x+T+1) = f(-x +T +1) = \rho (x)$, where $\rho (x) $
is a fixed smooth function on $[0,1]$.

Denote

$$ A_T = \{ (x,y) \in \Sigma_r , \; \; |x| <T+1 \} $$

Let $u_T$ be a solution in $A_T$ of the Dirichlet problem

$$ \De u_T =0 \;\;\;\; in \;\;\; \; A_T, $$
$$  u_T = f_T u \;\;\; on \;\;\;  \pa A_T $$
Then

$$ \int_{A_T} ( (u-u_T)^2 + | \nabla ( u -u_T) |^2) dxdy \leq C ,$$
where constant $C$ is independent of $T$. Thus applying Lemma \ref{SN-1} function $u_T$ and sending $T$ to $\infty $ we get the desirable inequality.

 \medskip  

Denote 

$$ f(y) = y^2 \mathcal M ( h^2( \cdot , y)) $$
From inequalities \eqref{SN2.2} we get

\begin{lemma}\label{sub} Function $f$ is absolutely monotonic on the segment  $0\leq y \leq h $.  For $ n \in \N$,

$$ f^{ (n)} (y) \geq 0 $$
\end{lemma}

 \medskip  

Notice, that by Bernstein's theorem absolutely monotonic functions have holomorphic extensions.

\begin{lemma}\label{VX} Let $H$ be a pluriharmonic function in $\Om_r, \, r<1,$ vanishing on the real subspace. Then
$$  \int_{\pa \Om_r} H { \pa H \over \pa n } d\si \leq 4  \int_{\Om_r} |\nabla_X H |^2 dxdy 
+ { 1 \over r } \int_{\pa \Om_r} H^2 d\si $$
\end{lemma} 

{\it Proof.} 
Let $y\in S^2\ss Y \ss \C^3$. Define
on the strip $\Sigma$ function $g(\xi, \eta )$,

$$ g(\xi, \eta ) = H(\xi i y + \eta y ) $$
Denote

$$ \de (\xi , \eta ) = |  H_\xi  ( \xi i y + \eta y) |^2  +  | H_\eta  ( \xi i y + \eta y) |^2 ,$$

$$ \ga (\xi , \eta ) = | H_\xi  ( \xi i y + \eta y) |^2  ,$$

$$ \beta (\xi , \eta ) = | H ( \xi i y + \eta y) |^2  $$
and

$$ D_1 ( \xi) = \int^r_0 \de ( \xi , \eta) d \eta ,$$

$$ D_2 ( \xi) = \int^r_0 \ga ( \xi , \eta) d \eta ,$$

$$ D_3 ( \xi) = \beta ( \xi , r) d \eta $$

 Consider
now the flow $i { \pa / \pa y }  $ on $M_y = M+y$. For almost all $y$ on $S^2$ the flow $i { \pa / \pa y } $ is a uniquely ergodic flow on $M$, 
its trajectory is dense on $M$ and for any continuous function on $M$ its time-average over any trajectory is equal to space-average over $M$. We denote by $\mathcal E \ss \S^2$ the set of $y$ for which the flow $i { \pa / \pa y } $ is an ergodic flow on $M$,   Thus since $V$ a pluriharmonic function from \eqref{SN2.4} we have

 \beq\label{5.2}  \mathcal M ( D_1) = { 1 \over |M |}  \int_{M_y} H H_y dx \eeq
By Lemma  \ref{SN-2}
 
$$  \mathcal M (D^1) \leq   \mathcal M (D^2) + \frac 1r  \mathcal M (D^3) $$
Hence

$$  \int_{M_y} H H_y dx \leq   \int^r_0  \int_{M_{ \eta y }}| \nabla_X H   |^2 dx d\eta + \frac 1r  \int_{M_{ ry }}|  H   |^2 dx  $$
Since all integrals in the last inequality are continuous functions of $y \in \S^2 $ and

$$ \int_{ \pa \Om_r} HH_y d\si =  \int_{\Om_r } | \nabla H |^2 dxdy $$
then integrating the last  inequality over $S^2 $ we get

 \beq\label{5.3}   \int_{\Om_r} { r^2 \over |y|^2 } | \nabla_X H |^2 dxdy + \frac 1r   \int_{\pa \Om_r}  |  H |^2 dxdy  \geq  \int_{\Om_r} | \nabla H |^2 dxdy \eeq
$0<r<1$.
Since $| \nabla_X H |^2 $ is a subharmonic function in $\Om $ and $|y|^{-1}$ is a harmonic function in $\Om$ we have the following inequalities for
$0<r<R<1$

$$ 0 \leq \int_{\Om_R \setminus \Om_r} \left( { 1 \over |y| } - \frac 1R \right) \De | \nabla_X H |^2 dxdy = $$
$$ = - { 1 \over r^2} \int_{ \pa \Om_r }  | \nabla_X H |^2 d\si + { 1 \over R^2} \int_{ \pa \Om_R }  | \nabla_X H |^2 d\si -  \left( { 1 \over r } - \frac 1R \right)  \int_{ \pa \Om_r } 
{\pa \over \pa r}   | \nabla_X H |^2 d\si $$
$$ = - { 1 \over r^2} \int_{ \pa \Om_r }  | \nabla_X H |^2 d\si + { 1 \over R^2} \int_{ \pa \Om_R }  | \nabla_X H |^2 d\si -  \left( { 1 \over r } - \frac 1R \right)  \int_{  \Om_r } 
\De | \nabla_X H |^2 dxdy$$
$$ \leq - { 1 \over r^2} \int_{ \pa \Om_r }  | \nabla_X H |^2 d\si + { 1 \over R^2} \int_{ \pa \Om_R }  | \nabla_X H |^2 d\si $$
Hence

$${ 1 \over r^2} \int_{ \pa \Om_r }  | \nabla_X H |^2 d\si \leq { 1 \over R^2} \int_{ \pa \Om_R }  | \nabla_X H |^2 d\si $$
for $0<r<R $. Therefore 

$${ 1 \over r^2} \int_{ \pa \Om_r }  | \nabla_X V |^2 d\si $$
is an increasing function of $r$. 
Hence

$$ \int_{  \Om_{ r/2 } } { 1\over |y|^2 } | \nabla_X V |^2 dxdy  \leq  \int_{ \Om_r - \Om_{r/2} } { 1 \over |y |^2} | \nabla_X V |^2 dxdy $$
Thus the lemma follows from the last inequality  and \eqref{5.3}

 \medskip  

From Lemma \ref{sub}, using the same argument as in the proof of the last lemma, we get 

\begin{lemma}\label{mon} Let $H$ be a pluriharmonic function in $\Om_r, \, r<1,$ vanishing on the real subspace. Let
$$ F(r) =   \int_{\pa \Om_r} H^2 d\si $$
Then
$$ F^{(n)} (r) \geq 0 ,$$
where $0\leq y \leq r , \, n \in \N $.
\end{lemma}

 \medskip  

 Let $w=u+iv$ be a holomorphic function in $\Om_r$, real on the real subspace.  Then

   \beq\label{SN2.3} \int_M \sum_j\left( { \pa v(x,0) \over \pa y_j } \right)^2 dx   = \int_M \sum_j\left( { \pa u(x,0) \over \pa x_j } \right)^2 dx  \eeq
   and
   
      \beq\label{SN2.5} \int_M \sum_j\left( { \pa^3 v(x,0) \over \pa y^3_j } \right)^2 dx   = \int_M \sum_j\left( { \pa^3 u(x,0) \over \pa x^3_j } \right)^2 dx  \eeq
      Since
      
 $$ || u(\cdot , 0) ||^3_{H^1(M)} \leq C || u(\cdot ,0) ||^2_{L^2(M)}     \left( \sum_j ||  { \pa^3 u(x,0) \over \pa x^3_j }  ||_{L^2(M) } \right) $$
 Therefore from \eqref{SN2.3}, \eqref{SN2.5} we have

  \beq\label{SN2.6}\sum_j ||  { \pa^3 u(x,0) \over \pa y^3_j }  ||_{L^2(M) } \geq C || u(\cdot , 0) ||^3_{H^1(M)} / || u(\cdot ,0) ||^2_{L^2(M)}    \eeq

 \medskip  

Now we  consider problem \eqref{NS} - \eqref{CP} with the initial data $w^0$ in $L^p(M)$  for $p>3$. Denote 
$ A = || w^0 ||_{L^p( M)} $. We assume that for all $t\geq 0$ external force $f( \cdot , t) $ is a divergence-free and real-analytic in the spatial variables
with the analyticity radius $r>0$. Let $f+ig$ be external force's analytic extension to $\Om_r$ and we assume

 \beq\label{ExF}   \sup_{t\geq 0, |y|<r}  \; ( | f( \cdot, y, t) | + | g( \cdot, y, t) | ) = \ga < \infty \eeq

Developing the Kato approach to mild solutions, [K], Grujić and Kukavica proved, [GK], that the problem \eqref{NS} - \eqref{CP} has a classical solution $w(x,t)$
on $M\times (0,T) $, $w\in C([ 0,T]) L^p(M) $,

$$ T = \min \left( { 1 \over  Cp^2 A^{2p/( p-3) }}, {1 \over Cp^2A^2 } , { A \over C \ga } \right)  $$
$C= C( p, a,  \nu ) > 0$, such that $w( \cdot , T )$ has a holomorphic extension $w=u+iv$ in $\Om_R $,

$$ R = T^{1/2} /C $$
moreover,

$$ || u( \cdot , y, T) ||_{L^p(M)} +  || v( \cdot , y, T) ||_{L^p(M)}  \leq CA $$
for $(x,y)\in \Om_R$. Thus

$$ || w( \cdot , T ) ||^2_{L^2 (  \Om_r) } \leq C r^3$$
$C= C( p,a, \nu , A, \ga ) > 0, \,  0<r<R$. Since $v$  vanishes on the real subspace of $\C^3$ 

$$ || v( \cdot , T ) ||^2_{L^2 (  \Om_r) } \leq C r^5$$

Therefore if $L^p$-norm of the initial data of problem is bounded by $A$ we may assume without loss that
the initial data $w_0$ has a holomorphic extension in $\Om_R, \, R=R( p,A, \ga)$.

\medskip Let $w(x,t)$ be a weak Leray-Hopf solution of problem \eqref{NS} - \eqref{CP} with the initial data $w_0 \in L^p, \, p>3 $. From [GK], [S], [G]
follows the existence of $t_0>0$ such that $w$ is a classical solution on $M\times (0, t_0)$. If $w$ is not a classical solution on $M \times (0, \infty )$ then there is 
a maximal interval $(0,t_1), \, t_0<t_1< \infty$, for which $w$ is a classical solution on $M\times (0, t_1) $.  Leray in [Le] called the number $t_1$
``époque de irrégularite".

\begin{theorem} \label{L}  Let $w(x,t)$ be a weak Leray-Hopf solution of problem \eqref{NS} - \eqref{CP} with the initial data $w_0 \in H^1(M) $. Let
$0<t_1 <\infty $ be an ``époque de irrégularite" for $w$. Then

$$ ||\nabla w ||_{L^2 ( M)} \geq {C \nu^{3/4} \over ( t_1-t)^{1/4} } $$
for $t<t_1$.
\end{theorem}

  Theorem \ref{L} is essentially due to Leray [Le], see [S], [G]. As one can see from the proof  in [G], that the theorem holds with non-zero external forces
  $f$ if we assume that the norm of $f$ is uniformly bounded in $H^1( M) \times (0, t_1)$. 
  
\smallskip

Let $w$ be a weak Leray-Hopf solution of problem  \eqref{NS} - \eqref{CP} with the initial data $w^0 \in L^p(M), \, p>3 $ and let $t_1 >0$ be its
``époque de irrégularite". From the two last theorems follows that $w(x,t) $ is real-analytic for $t\in (0,t_1)$ with the radius of analyticity $\rho (t) >0$.

  {\it Proof of Theorem \ref{NN}. } First we prove the existence of the classical solution of problem   \eqref{NS} - \eqref{CP}.  for the initial data $w^0$ in $L^p( M), \, p>3 $. Let $w$ be a weak Leray-Hopf solution this problem. Assume by contradiction that $w$ is not
  real-analytic and hence not a classical solution. Then there exists an ``époque de irrégularité" $0 < t_1 < \infty$. 
  
  \smallskip 

There is a constant $N>0$ depending on $f, w_0$ such that for $t\in (0, t_1) $

$$ || w(\cdot ,t) ||_{L^2(M)} \leq N $$

\medskip Denote by $W=W=U+iV =W^m= U^m+ iV^m$ a solution of Faedo-Galerkin problem  with the initial data $\pi_m w^0$, where $m \in \N $ is a sufficiently
large constant, which we define later. Denote

$$ E( r,t) = \int_{\Om_r} |V |^2 dxdy ,$$

$$ e( r,t) = \int_{ \pa \Om_r} |V |^2 d\si ,$$

$$ D( r,t) = \int_{\Om_r} | \nabla V |^2 dxdy $$
Then 

$$e =E_r ,$$
$$e_r = 2D + 2e/r ,$$
 \beq\label{Err}   E_{rr} -  2 E_r /r  = 2 \int_{\pa \Om_r}( V, { \pa V \over \pa n } )d\si \eeq

\smallskip  

Define

$$ y_k (x) = { k \over 120 } x^5 +x^{41/8} $$

\smallskip  

 For sufficiently large $k$ and $x\in (0, 1 /k^{3/2}) $ the following inequality holds

 \beq\label{SN2.7}  { 1 \over x^{3/2}} \sqrt { y'_k (x) }<   x^{1/6} \eeq

Assume that

$$ {\pa^5 E(0,t_0) \over \pa r^5 } \geq k ,  $$
then from  Lemma \ref{mon} and inequality \eqref{SN2.6} follows, that for sufficiently large $k$ 

 \beq\label{SN-1}    E_r ( 1 / k^{3/2}, t_0) >  y'_k ( 1 /k^{3/2} ),  \eeq
Assume that

$$ {\pa^5 E(0,t_0) \over \pa r^5 } = k   $$
Then there is an $\varepsilon >0$ such that for $0<x< \varepsilon $ the following inequality holds

 \beq\label{SN-6}    E ( x, t_0) <  y_k ( x )  \eeq
Notice, that the inequality

$$ {\pa^5 E(0,t_0) \over \pa r^5 } \geq 2k ,  $$
implies

 \beq\label{SN-2}    E ( x, t_0) >  y_k ( x ),  \eeq

\medskip  Taking the difference between the equalities of Theorems \ref{FGE} and \ref{FGE2} we get

$$ 2 \nu \int_{ \Om_r} | \nabla_X v |^2 dxdy  = - \int_{\Om_r } { \pa |v|^2 \over \pa t} dxdy  - \frac 1r \int_{\pa \Om_r } \tilde v |v|^2 d\si + 2 \int_{ \Om_r }  v g dxdy $$
From \eqref{Err}  and Lemma \ref{VX} for $t_2 < t< t_1$,

$$ E_{rr} -2E_r / r \leq 8 \int_{ \Om_r} | \nabla_X v |^2 dxdy + { 2 \over r}  \int_{\pa \Om_r} |  v |^2 d\si  $$
$$ = - { 4 \over \nu } E_t  - { 4 \over \nu r } \int_{\pa \Om_r } \tilde v |v|^2 d\si + {8 \over \nu} \int_{ \Om_r }  v g dxdy   + { 2 \over r} E_r(t) $$

By Lemma \ref{BasicHol2} and the bounds for external forces \eqref{ExF} 
we get  the following {\it nonlinear differential inequality of the parabolic type
for the complex energy }

 \beq\label{parce}  E_{rr} + \frac {4}\nu E_t   \leq { C_0 \over \nu r ^{3/2} }  E_{rr} E_r^{1/2} + \frac 4r E_r + Cr^3E^{1/2}_r   \eeq
where $0<r<1$   $C>0$ depends on $p,a, f$. Notice, that  in \eqref{parce} for small $E_r$ {\it time has the backward direction},  comparing it with the initial Navier-Stokes equations.
Similarly, backward time appears in the diffusion processes like the heat equation, or, more general, a parabolic system in a complex domain if one rewrite it  as an equation in imaginary coordinates.

Assume that $E_r( r_0, t_0) \geq 0$. Then

 \beq\label{Errb} E_{rr}(r_0,t_0) \leq { b \over  r ^{3/2}  }   E_{rr}(r_0,t_0) E_r^{1/2} (r_0,t_0) + \frac 4r E_r(r_0,t_0) + Cr^3E^{ 1/2}_r(r_0,t_0) \eeq
where $b=C_0 / \nu , \, C=C(p,a, \nu , \ga )>0 $.

From inequality \eqref{SN2.7}  follows that for any $C,b>0$ there is $ K >0, \, K=K(C,b)$ such that for $k>K$ function $y_k(x) $ is a subsolution of a differential inequality on 
$(0, \de ) ,  \, \de = 1/ k^{3/2} $,

 \beq\label{diffin}   y''_k> { b\over x^{3/2}   }  (y''_k)(y'_k)^{1/2} + \frac 4x y'_k + C x^3(y'_k)^{1/2} \eeq

 For constant $K=K(C,b)$ by Theorem \ref{L} and Proposition \ref{pr}  follows that there are $t_2,t_3, \, 0<t_2<t_3 <t_1$ and sufficiently large constant $m$ such that for
 Faedo-Galerkin approximations $W=W=U+iV =W^m= U^m+ iV^m$ the following inequalities hold
 
 $$ {\pa^5 E(0,t_2) \over \pa r^5 } =  K  , $$
 
  $$ {\pa^5 E(0,t_3) \over \pa r^5 } = 2 K   $$
 and for $t_2 <t<t_3$
 
  $$ {\pa^5 E(0,t) \over \pa r^5 } >  K   $$
  
  For $t_2 < t < t_3 $ from the inequality \eqref{SN-1} follows
  
   \beq\label{SN-3}    E_r ( \de , t) >  y'_k ( \de ),  \eeq

 Define $\Phi (r,t) $ on $ (0, \de ) \times ( 0 , \tau ), \, \tau = t_3-t_2 $,

$$ \Phi (r,t) = y(r) - E( r,t -t_2) $$
Then

$$\Phi (0,t) =0, $$
from inequality \eqref{SN-3}

 \beq\label{SN-4} \Phi' (\de,t) > 0 \eeq

Denote

$$ \phi (t) = \sup_{ 0<r< \al /2} \Phi ( r,t) $$
Then from inequality \eqref{SN-6}

$$ \phi (0) > 0 $$ 

From inequality \eqref{SN-2}

$$ \phi (\tau) < 0 $$
Hence there exists $T \in ( 0, \tau  )$ such that

$$ \phi' (T) < 0 $$

Let $0  \leq r_0 \leq \de $ be such that

$$ \Phi( r_0, T) =\phi (T) $$

From inequality \eqref{SN-4} follows

$$ 0 < r_0 < \de $$
Therefore,

 \beq\label{E1}  E_r ( r_0,T ) =y' (r_0) ,\eeq
 
 \beq\label{E2}  E_t (r_0, T) \geq 0 \eeq
and since at the point $r_0$ function $ \Phi ( r, T)$ attains its supremum

 \beq\label{E3} E_{rr} (r_0 ,T ) \geq y'' (r_0 ) \eeq

 \includegraphics[width=4.5in]{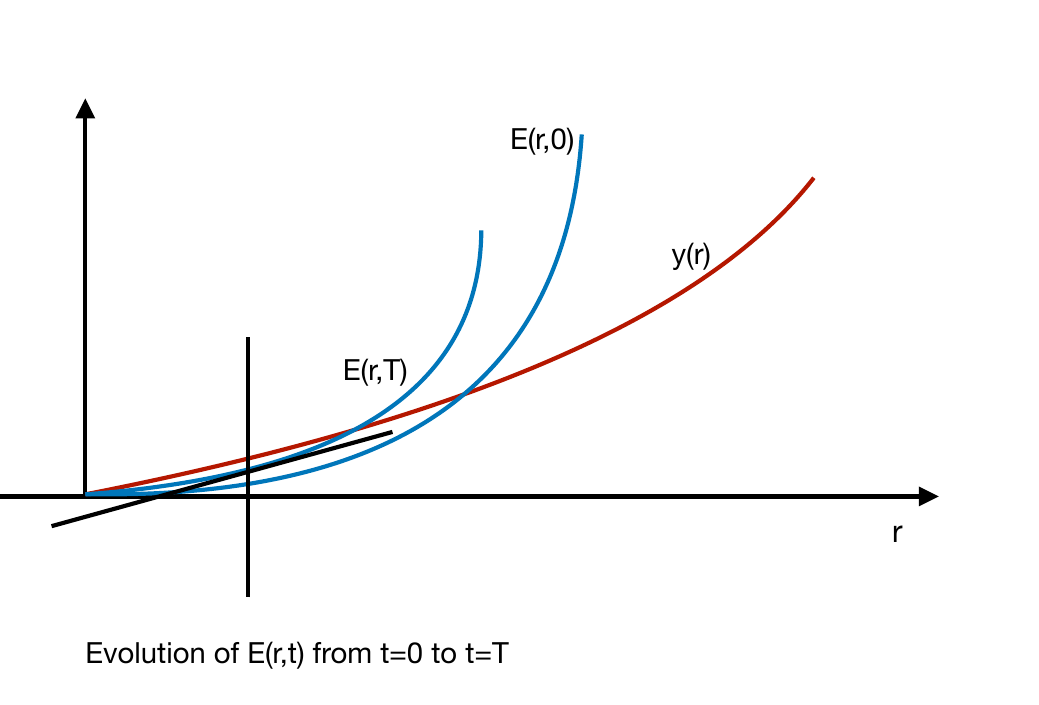}  

 Denote

 $$ S(r, y'', y') = 1- { b \over r^{3/2}  } (y')^{1/2}- \frac 4r {y' \over (y'')} - Cr^3 {(y')^{1/2} \over y'' }$$ 
 Then we can rewrite the inequality  \eqref{diffin} as
 
 $$ S( r, y''(r), y'(r)) > 0 $$
 
 Since $y'', y' > 0, \, S $ is an increasing  function of $y''$. Then from \eqref{E1}, \eqref{E3} follows
 
  $$ S( r_0, E''(r_0,T), E'(r_0, T)) > 0 $$
  Thus
 
 $$ E_{rr}(r_0,T)  > { b \over  r ^{3/2} }   E_{rr}(r_0,T) E_r^{1/2}(r_0,T) $$$$ + \frac 4r E_r(r_0,t_0) + Cr^3E^{1/2}_r(r_0,T)  $$

The last inequality and \eqref{E2}  contradict  \eqref{Errb}.

Therefore $w$ is a classical solution of the Navier-Stokes equations in $M \times (0, \infty)$, 
By  uniqueness theorem for solutions of the Navier-Stokes equations in $L^3$,  [G],  $w$ is a unique solution of the initial value problem \eqref{NS} - \eqref{CP} with a given
 initial data $w^0$.
 

 
  



Assume now that $w(x,t) $ is a weak Leray-Hopf solution of problem \eqref{NS} - \eqref{CP}  with the initial data $w^0$ in $L^2( \pa \Om) $. Then 
$ w\in L^2 ( \R_+ ) H^1 ( M) $. Thus there is a sequence $t_i >0,\, t_i \to 0 $ as $i\to \infty $ such that 
$ w( \cdot , t_i) \in H^1( M) \in L^6 (M) $. Then  from  proved above regularity of the solutions for $p>3$ it follows that for any $t_i$ $w$ is a classical solutions on $M\times ( t_i, \infty )$
and Theorem \ref{NN} follows.

\section{Discusion  }

We discuss here miscellaneous issues on possible extensions of the main result.

{\bf 1. How close are blowing up systems to Navier-Stokes equation?}

Regarding  regularity results for a concrete nonlinear differential equation, the first rising question, is it possible to include the given equation into a class
of equations for which the same regularity holds, for instance, for the equations with a similar nonlinearity, etc. 
We discuss below the blowing examples for models closely related to the three-dimensional Navier-Stokes system.  

Li and Sinai, [LS] , [BLS] gave an example of a singular solution for a system generated from complex valued solution of the Navier-Stokes equation, defined on the
real space. Though the system is similar to usual 
Navier-Stokes equations, its solutions do not satisfy the energy inequality and can develop a singularity. Recently Sverak has shown that quaternion
values solutions of one-dimensional Burgers equation develop a singularity, [S].  Notice, that complex or quaternionic  valued  Navier-Stokes equations are different from the Navier-Stoces equations in complex space. We consider the complexification of the vector space of the arguments of the solutions,
 $\R^3 \otimes \C$. The extension of the field of  arguments of the equations decrease the set of its solutions  since
it makes the system overdetermined,  it gives additional equations which can help to prove the regularity
of the solution. If instead of arguments  we extend the field of values of the solution, it will give more room for the singular solutions. See in this
connection, [NTV].

Tao in [T] suggested a model,  he called ``averaged Navier-Stokes equation", which solutions satisfy the energy inequality. Let $u$ be a sufficiently
regular solution of Navier-Stokes equations \eqref{NS} on  $\R^3 \times (0, \infty) $. Then one can rewrite \eqref{NS} as

$$ \pa_t u - \nu \De u + P(  u \cdot \nabla u ) =0 $$
where $P$ is the orthogonal projection of vector fields onto solenoidal vector fields,

$$ P u_i = u_i - \De^{-1}  \pa_i \pa_j  u_j  $$
Operator $P$ satisfies the identity,

$$ (Pv,v)=0 $$
for all solenoidal vector fields $v \in L^2(\R^3) $

Loosely speaking Tao proved the following theorem. There exists a pseudo-differential operator $D$ of order $0$ satisfying the identity

$$ ( Dv,v ) = 0 $$
for all  solenoidal vector fields $v\in L^2 (\R^3) $ such that for the Cauchy problem in $\R^3$ for the equation 
$$ \pa_t u - \nu \De u + D(  u \cdot \nabla u ) =0 $$
there is a smooth initial data  $u^0= u ( \cdot , 0) $ such that solution $u$ is blowing up at a finite time $t_0 >0$.

From the Tao's result it follows that  all estimates to solution of the  Navier-Stokes equations which could be obtained by ``standard means of the functional analysis" on $\R^3$  are also true for the  ``average Navier-Stokes" equations. Thus the regularity of the Navier-Stokes equations can be only a consequence of its concrete particular form. 

There are numerous examples of singular solutions for systems coupling Navier-Stokes equations with some other equations, like conservation of mass equation for compressible Navier-Stokes equations, Maxwell's equations in magnetohydrodynamics, etc.  

\smallskip  {\bf Euler equations. } The dissipation term in Navier-Stokes equations plays a crucial role in the proof of Theorem \ref{NN}. The theorem is not valid for weak solutions of the Euler equations as it follows from the results of Shnirelman [Sh] and De Lellis -  L. Székelyhidi Jr. [DS] . A general smooth solution of the Euler equations has no 
 holomorphic extension. However, for the Euler equations in Lagrangian
coordinates there is the real-analyticity of individual trajectories,  [CKV], and one can try to apply the analytic framework to the individual trajectories, or to consider the
Arnold-Euler equation, [AK], on  the complexification of the infinite-dimensional Lie group of volume preserving diffeomorphisms of a manifold.


\smallskip {\bf Spatial dimension. }  We consider the problem of regular solution to Navier-Stokes equations only in dimension $3$. In dimensions $\geq 4$
our approach doesn't give the desirable outcome. The regularity of weak solutions of initial boundary value problem in a two-dimensional bounded domain with zero Dirchlet boundary data, in absence of external 
forces was first obtained by Leray, [Le2]. The problem was completely settled in dimension 2 
by Ladyzhenskaya, Lions and Prodi, [L1], [L2], [LP]. 

\smallskip {\bf Very weak solutions. } Theorem of Lions and Masmoudi [LM] shows that for $p>3$ and non-forced Navier-Stokes equations Theorem \ref{NN} will be valid if instead of
Leray-Hopf weak solutions we take very weak solutions without assuming any type of energy inequality. Though, further enlargement of the  set of weak solutions leads to non-forced weak solutions with
the energy increasing in time, [BV].


\smallskip {\bf Uniqueness of the weak solutions. } In $L^p (\R^3), \, 2\leq p < 3$, there is probably no uniqueness to the weak Leray-Hopf solutions
of Navier-Stokes equations, [JS]. It is not clear, are the arguments of [JS] do work in a compact case, for instance at the torus.  The influence of the
``infinity" in  parabolic Cauchy problems is very strong. We recall here a classical counterexample to the uniqueness of the Cauchy problem for the heat equation, [Ty]. For Navier-Stokes equations there is no uniqueness for the Cauchy problem in $L^\infty$, [Se].  Though, there is a uniqueness in $BMO^{-1}( \R^3) $ for mild solutions of the Navier-Stokes
equations,  [KT]. A non-uniqueness result for the forced Navier-Stokes equations with a singular in time force is given in a recent paper, [ABC].

\smallskip {\bf Global in time estimates for the solutions.} Assume that $w$ is a solution of \eqref{NS} - \eqref{CP} in absence of exterior forces, $f\equiv 0$,
and  the initial data $w^0 \in H^1(M)$. As it follows from [He], [G] there is a constant $C>0$ such that if

$$ t > { C\over \nu^{-5} } || w^0 ||^4_{H^1(M)} $$
then

$$ || w(\cdot , t) ||_{H^1(M)} \leq 1 $$
Assume now that $w^0 \in L^p(M), \, p>3$, and $|| w^0 ||_{L^p} =A$. Then from [GK], Theorem \ref{NN} and the uniqueness theorem, [LM] it follows that there is a constant $C=C(p,A, \nu ) >0$ such that 
for $t\in (0, \infty )$ holds the estimate

 \beq\label{DC}  || w(\cdot , t) ||_{L^p(M)} \leq C \eeq

Of course, for the solutions of Navier-Stokes equations with non-zero external forces there are no such uniform in time estimates. However, if $f$ is uniformly bounded in
$\Om_\rho \times (0, \infty ), \, \rho >0$, we suggest as a conjecture that inequality \eqref{DC} holds.

\smallskip {\bf Sharp constant $\al $.} We expect that the sharp constant $\al $ in inequality \eqref{SN12} is $ < 2$.  It's possible to construct examples of holomorphic solenoidal vector fields showing that for $\al <3/7$ the inequality is not valid.  

\smallskip {\bf External forces. } It is impossible to extend directly the proof of Theorem \ref{NN} to include the Navier-Stokes equations with smooth external
forces, or to consider other initial-boundary value problems. In the future work we plan  to address
these issues. 

\smallskip {\bf Stationary Navier-Stokes equations. } For independent of time solutions of the Navier-Stokes equations the differential inequality for
the complex energy \eqref{parce} takes a simple form of an ordinary differential inequality of the first  order for function $e(r)$.

Let $w(x)$ be a solution on $M$ of the Navier-Stokes equations

$$ - \nu \De w + w \cdot \nabla w -  \nabla p = f       $$
$$ \div \, w =0 $$

Assume that $w$ has a holomorphic extension in $\Om_R, \, R>0$. Then on $(0,R)$ holds the inequality 

$$ e_{r}  \leq { C \over \nu r ^{3/2} }   e_{r}  e^{1/2} + \frac 4r e  + Cr^3e^{1/2}$$

The last inequality contains an implicit relation between the Fourier spectrums of the external force $f$ and of the solution $w$.


 

 
  


  \bigskip

\centerline{REFERENCES} 

 \bigskip
 \medskip\noindent [AF] R.A. Adams, J.J.F. Fournier, {\it Sobolev Spaces}, Academic Press, 2003.
 
  \medskip\noindent [ABC] D. Albritton, E. Brué, M. Colombo, {\it Non-uniqueness of Leray solutions of the forced Navier-Stokes equations}, Annals Math. 196 (2022), 415-455
 
\medskip\noindent [AK] V.I. Arnold, B.A. Khesin, {\it Topological Methods in Hydrodynamics}, Springer, Berlin, 1998.


\medskip \noindent  [BLS] C. Boldrighini, D. Li, Ya.G. Sinai, {\it  Complex singular solutions of the 3-d Navier-Stokes equations and related real solutions}
J. Stat. Phys. 167 (2017), 1-13

\medskip \noindent  [B] M. Bramanti, {\it On the proof of Hörmander's hypoellipticity theorem}, Le Matematiche, 75 (2020), 3-26.

\medskip \noindent  [BB] M. Bramanti,  L. Brandolini, {\it Schauder estimates for parabolic nondivergence operator of Hörmander type}, J. Diff. Eq. 234
(2007), 177-245


\medskip \noindent  [BV] T. Buckmaster,  V. Vicol, {\it Nonuniqueness of weak solutions to the Navier–Stokes equation}, Ann. of Math. (2), 189 (2019), no. 1, 101-144.

\medskip \noindent  [CKV] P. Constantin, I. Kukavika, V. Vicol, {\it Contrast between Lagrangian and Eulerian analytic regularity properties of Euler equations}
 Ann. Inst. H. Poincaré C Anal. Non Linéaire 33 (2016), no. 6, 1569-1588.
 
 \medskip \noindent  [C] C. Corduneanu, {\it Almost periodic functions}, Chelsea Publishing Company, N.Y., 1989 
 
 \medskip \noindent  [D] B.E.J. Dahlberg, {\it Weighted norm inequalities for Lusin area integral and the nontangential maximal function for functions harmonic in a Lipschitz domain}, Studia Math. 67 (1980), 297-314


\medskip \noindent  [DS] C. De Lellis and L. Székelyhidi Jr., {\it Dissipative continuous Euler flows}, Invent. Math. 193 (2013), no. 2, 377-407


 

\medskip\noindent [ESS] L. Escauriaza, G. Seregin, V. Sverák, {\it $L^{3,\infty} $ -solutions of Navier- Stokes equations and backward uniqueness },
Russian Math. Surveys 58 (2003), 211-250.

\medskip\noindent [FS] C.L. Fefferman, A. Sánchez-Calle, {\it Fundamental solutions for second order subelliptic operators}, Annals Math. 124 (1986), 247-272


\medskip\noindent [F] M. Freeman, {\it The Levi form and local complex foliations}, Proc. Am. Math. Soc. 57 (1976), 369-370


\medskip\noindent [G] G.P. Galdi {\it An Introduction to the Navier-Stokes Initial-Boundary Value Problem } Fundamental directions in mathematical fluid mechanics, 1-70, Adv. Math Fluid Mech., Birkhäuser, Basel, 2000.
 
\medskip\noindent [GK] Z. Grujić, I. Kukavica {\it  Space Analyticity for the NavierStokes and Related Equations with Initial Data in $L^p$ }, J. Func. An. 152
(1998), 447-466. 




\medskip \noindent [He] J.G. Heywood, {\it The Navier-Stokes Equations: On the Existence, Regularity and Decay of Solutions}, Indiana Univ. Math. J., 29
(1980), 639-681 

\medskip \noindent [HP] M. Hieber, J. Prüss, {\it Heat kernels and maximal $L^p -L^q$ estimates for parabolic evolution equations }, Comm. Part. Diff. Eq. 22 (1997),
1647-1669.
 
\medskip \noindent [Ho] E. Hopf, {\it \"Uber die Anfganswertaufgabe f\"ur die Hydrodynamischen Grundgleichungen,} Math. Nachr. 4(1950), 213--231.

\medskip \noindent [H1] L. Hörmander {\it Hypoelliptic second-order differential equations},  Acta Math. 119, (1967), 147-171 

\medskip \noindent [H2] L. Hörmander {\it An Intoduction to Several Complex Variables }, Van Nostrand, Princeton, NJ, 1966 

\medskip \noindent  [JK] D. Jerison, C.E. Kenig, {\it The inhomogeneous Dirichlet problem in Lipschitz domains}, J. Functional  An. 130 (1995), 161-219

\medskip \noindent  [JS] D. Jerison, A. Sánchez-Calle, {\it Estimates for the heat kernel for a sum of squares of vector fields}, Indiana J. Math., 35 (1986),
835-854

\medskip \noindent  [JS] H. Jia, V. Šverák, {\it Are the incompressible 3d Navier-Stokes equations locally ill-posed in the natural energy space?}, J. Func. Anal. 268(12), (2015), 3734-3766 

\medskip \noindent [KT] H. Koch, D. Tataru, {\it Well Posedness for the Navier-Stokes equations }, Adv. Math. 157 (2001), 22-35

\medskip \noindent [K] T. Kato, {\it Strong solutions of the Navier-Stokes equations in $\R^m$ with applications to weak solutions}, Math. Z. 187 (1984), 471-480 

\medskip \noindent [Ko] J.J. Kohn {\it Pseudo-differential operators and hypoellipticity}, (Proc. Sympos. Pure Math., Vol. XXIII, Univ. California, Berkeley,
Calif., 1971), pp. 61-69. Amer. Math. Soc., Providence, R.I., 1973.


\medskip \noindent [KS] S. Kusuoka, D. Stroock, {\it Applications of the Malliavin calculus. II}, J. Fac. Sci. Univ. Tokyo Sect. IA Math. 32 (1985),  1-76


\medskip \noindent  [L1] O. Ladyzhenskaya, {\it Solution in the Large of the Nonstationary Boundary-Value Problem for the Navier-Stokes System for the Case of Two Space Variables}, Doklady Akad. Nauk. SSSR, 123, (1958) 427 (in Russian)

\medskip \noindent  [L2] O. Ladyzhenskaya, {\it The Mathematical Theory of Viscous Incompressible Flow,} Gordon and Breach, NY, 1969.

\medskip \noindent  [LSU] O.A. Ladyzhenskaya, V.A. Solonikov, N.N. Uralćeva, {\it Linear and Quasilinear Equations of Parabolic Type }, AMS, Providence,
R.I., 1968. 

 
\medskip\noindent  [Le] J. Leray, {\it Sur le Mouvements d'un Liquide Visqueux Emplissant l'Espace,} Acta Math. 63(1934), 193--248.

\medskip\noindent  [Le2] J. Leray, {\it Essai sur les Mouvements Plans d'un Liquide Visqueux que Limitent des Parois}, J. Math. Pures Appl. 13 (1934), 331-418

\medskip \noindent  [LR] P.G. Lemari\'e-Rieusset, {\it The Navier-Stokes Problem in the 21 Century}, CRC Press, Evry, 2016.

\medskip\noindent  [LM1] J.-L Lions, E. Magenes,  {\it Problems aux limites non homogènes v.1}, Dunod, Paris, 1968 

\medskip\noindent  [LM2] J.-L Lions, E. Magenes,  {\it Problems aux limites non homogènes v.2}, Dunod, Paris, 1968 

\medskip\noindent  [LP] J.-L Lions, G. Prodi,  {\it Un Théorème d'Existence et Unicité dans les Equations de Navier-Stokes en dimension 2}, C. R. Acad. Sci. Paris,
248, (1959), 443-446.  

\medskip\noindent  [LM] P.-L Lions, N. Masmoudi, {\it Uniqueness of mild solutions of the Navier-Stokes system in $L^N$ }, Comm. Part. Diff. Eq. 26 (2001), 2211-2226 


\medskip \noindent  [LS] D. Li, Ya.G. Sinai, {\it  Blowups of complex solutions of the 3D Navier–Stokes system and renormalization group method }, J. Eur. Math. Soc. 10,  (2008) , 267-313.

\medskip \noindent  [LY] P. Li, S.-T. Yau, {\it On the parabolic kernel of the Schrödinger operator}, Acta Math. 156 (1986), 153-201 

\medskip\noindent [Ma] K. Masuda, {\it On analyticity and the unique continuation for solutions of the Navier-Stokes equation}, Proc. Japan Acad. 43 (1967),
183-186


\medskip\noindent [ML] A. Montanari, F. Lascialfari, {\it The Levi Monge-Ampère equation: smooth regularity of strictly Levi convex solutions}, J. Geom. Anal. 14 (2004),  2, 331-353



 
 \medskip\noindent [Mi]  M. Mitrea, {\it On Dahlberg's Lusin area integral theorem}, Proc. American Math. Soc. 123 (1995), 1449-1455  
 
 \medskip\noindent [NTV] N. Nadirashvili, V. Tkachev, S. Vlăduţ, {\it Nonlinear elliptic equations and nonassociative algebras}. Mathematical Surveys and Monographs, 200. American Mathematical Society, Providence, RI, 2014.
 
  \medskip\noindent [NSW] A. Nagel, E.M. Stein, S. Wainger, {\it Balls and metrics defined by vector fields I: Basic properties}, Acta Mathematica,
  155 (1985), 130-147
  
 
 \medskip\noindent [OR]  O.A. Oleinik, E.V. Radkevič, {\it Second order equations with nonnegative characteristic form}, Progress in Science, Mathematics Series,
 1969 (Russian)

\medskip\noindent [RS] L.P. Rothschild, E.M. Stein {\it  Hypoelliptic differential operators and nilpotent Lie groups},  Acta Math. 137,  (1977), 247-320

 \medskip \noindent [Sa] D.A. Salamon, {\it Parabolic $L^p - L^q $ estimates}, preprint ETH Zürich, 2017
 
  \medskip \noindent [Sh] A. Shnirelman, {\it Weak solutions with decreasing energy of incompressible Euler equations}, Comm. Math. Phys. 210 (2000), no. 3, 541-603

 
 
\medskip \noindent [Se] J. Serrin, {\it On the interior regularity of weak solutions of the Navier-Stokes equations,} Arch. Rat. Mech. Anal. 9(1962), 187--195.

  \medskip \noindent [So] F. Sommer, {\it Komplex-analytische Blãtterung reeller Mannigfalttigkeiten im $\C^n$}, Math. Ann. 136 (1958), 111-133 
  
  
    \medskip \noindent [St] R.S. Strichartz, {\it Analysis of the Laplacian on the complete Riemannian manifold}, J. Funct. Anal. 52 (1983), 48-79 
  

\medskip\noindent [Sv] V. Sverak {\it On singularities in the quaternionic Burgers equation}, Ann. Math. Qué. 46 (2022), 41-54

\medskip\noindent [T] T. Tao, {\it  Finite time blowup for an averaged three-dimensional Navier–Stokes equation},  J. Am. Math. Soc. 29(3), (2016),
601-674 

\medskip\noindent [Te] R. Temam, {\it Navier-Stokes equations}, North-Holland Publishing Company, Amsterdam-New York-Oxford, 1979

\medskip\noindent [Ty]  A.N. Tychonoff, {\it Uniqueness theorem for the heat equation }, Mat. Sb. 42 (1935), 199-216



\end{document}